\documentclass[a4paper,12pt]{amsart}
\usepackage[left=3.4 cm,top=2.4cm,bottom=3.4cm,right=3.4cm]{geometry}

\usepackage{amsfonts,amssymb,amsthm,amsmath,amssymb}
\usepackage{bbm}
\usepackage{xfrac}

\usepackage{xcolor}

\usepackage{physics}

\usepackage{enumitem}

\definecolor{green(munsell)}{rgb}{0.0, 0.66, 0.47}
\definecolor{BlueGreenn}{rgb}{0.3,0.5,0.8}
\definecolor{DB}{rgb}{0.3,0.3,0.3}
\definecolor{DOr}{rgb}{0.7,0.3,0.3}
\definecolor{DGr}{rgb}{0.3,0.7,0.3}
\definecolor{DBl}{rgb}{0.1,0.3,0.5}
\definecolor{arylideyellow}{rgb}{0.91, 0.84, 0.42}
\definecolor{burntorange}{rgb}{0.8, 0.33, 0.0}
\definecolor{chromeyellow}{rgb}{1.0, 0.65, 0.0}
\usepackage[hypertexnames=true, citecolor = burntorange, colorlinks = true, linkcolor = BlueGreenn, pdffitwindow = false, urlbordercolor = white,pdfborder={0 0 0}]{hyperref}%

\usepackage[T1]{fontenc}
\usepackage{microtype}

\usepackage{color}
\usepackage[utf8]{inputenc}

\usepackage{framed}
\usepackage{centernot}

\usepackage{tikz-cd} 
\usepackage{tkz-euclide}
\usepackage{mathtools}
\mathtoolsset{showonlyrefs=true}

\numberwithin{equation}{section}

\renewcommand{\phi}{\varphi}

\newtheorem{theorem}{Theorem}[section]
\newtheorem{proposition}[theorem]{Proposition}
\newtheorem{lemma}[theorem]{Lemma}

\newtheorem{corollary}[theorem]{Corollary}
\newtheorem{definition}[theorem]{Definition}

\newcommand{\ot}{\otimes}
\newcommand{\tp}[1]{^{\otimes #1}}    % tensorial power

\renewcommand{\op}{\oplus}

\newcommand{\scp}[2]{\langle #1, #2 \rangle}

\DeclareMathOperator{\Aut}{Aut}

\DeclareMathOperator{\Ad}{Ad}
\DeclareMathOperator{\dom}{dom}%\mathfrak{D}}

\DeclareMathOperator{\sgn}{sgn}

\newcommand{\Cl}{\mathbb{C}}
\newcommand{\Rl}{\mathbb{R}}
\newcommand{\Nl}{\mathbb{N}}
\newcommand{\Zl}{\mathbb{Z}}
\newcommand{\Zlt}{\mathbb{Z}_2}

\newcommand{\bom}{\boldsymbol{\omega}}
\newcommand{\bomcan}{\boldsymbol{\omega}^{\rm can}}
\newcommand{\bomcanb}{\boldsymbol{\omega}^{\rm can}_\beta}

\newcommand{\bomGibbsb}{\boldsymbol{\omega}^{\rm Gibbs}_\beta}
\newcommand{\balpha}{{\boldsymbol{\alpha}}}
\newcommand{\bA}{\boldsymbol{\A}}
\newcommand{\bunit}{\boldsymbol{1}}
\newcommand{\States}{{\mathcal S}}
\newcommand{\KMS}{{\States}_\beta}

\newcommand{\CO}{{\mathcal O}}
\newcommand{\CC}{{\mathcal C}}
\newcommand{\Ham}{H}%\mathscr{H}} %Hamilton operator

\newcommand{\id}{\mathbbm{1}}

\newcommand{\M}{\mathfrak{M}}

\newcommand{\B}{\mathcal{B}}
\newcommand{\T}{\mathcal{T}}

\newcommand{\Hil}{\mathcal{H}}

\newcommand{\Fock}{\mathcal{F_-(\Hil)}}

\newcommand{\A}{\mathcal{A}}

\newcommand{\Zen}{\mathcal{Z}}

\newcommand{\Om}{\Omega}
\newcommand{\om}{\omega}
\newcommand{\la}{\lambda}
\newcommand{\eps}{\varepsilon}
\newcommand{\te}{\theta}

\newcommand{\CAR}{{\rm CAR}(\Hil)}
\newcommand{\CARDyn}{{\rm{CAR}}(\Hil, U)}

\newcommand{\CARDyng}{({\rm{CAR}}(\Hil), \alpha^H, \gamma_G)}

\usepackage{mathrsfs}

\newcommand{\CP}{{\mathcal P}}



\usepackage{comment} % to exclude whole sections when I want to

\newif\ifshow % toggle true or false based on if want to hide section
\showfalse % hide the sections

\ifshow
  \includecomment{wrap}
\else
  \excludecomment{wrap} % anything wrapped in a {wrap} environment will be excluded
\fi
\title{KMS states on $\mathbb{Z}_2$-crossed products \linebreak and twisted KMS functionals}
\author{Ricardo Correa da Silva}
\address{Department Mathematik, FAU Erlangen-Nürnberg, ricardo.correa.silva@fau.de}
\author{Johannes Große}
\address{Department Mathematik, FAU Erlangen-Nürnberg, joh.grosse@fau.de}
\author{Gandalf Lechner}
\address{Department Mathematik, FAU Erlangen-Nürnberg, gandalf.lechner@fau.de}

\date{March 14, 2024.}

\begin{document}

\begin{abstract}
	KMS states on $\mathbb{Z}_2$-crossed products of unital $C^*$-algebras~$\mathcal{A}$ are characterized in terms of KMS states and twisted KMS functionals of $\A$. These functionals are shown to describe the extensions of KMS states $\omega$ on~$\mathcal{A}$ to the crossed product $\mathcal{A} \rtimes \mathbb{Z}_2$ and can also be characterized by the twisted center of the von Neumann algebra generated by the GNS representation corresponding to $\omega$.
	
	As a particular class of examples, KMS states on $\mathbb{Z}_2$-crossed products of CAR algebras with dynamics and grading given by Bogoliubov automorphisms are analyzed in detail. In this case, one or two extremal KMS states are found depending on a Gibbs type condition involving the odd part of the absolute value of the Hamiltonian.
	
	As an application in mathematical physics, the extended field algebra of the Ising QFT is shown to be a $\mathbb{Z}_2$-crossed product of a CAR algebra which has a unique KMS state.
\end{abstract}

\maketitle

\section{Introduction}

KMS states of $C^*$-dynamical systems provide on the one hand the general formalization of thermal equilibrium states \cite{BratteliRobinson:1997}, and are on the other hand natural generalizations of tracial states that are intimately connected to Tomita-Takesaki modular theory \cite{Takesaki:2003}. Notwithstanding their importance in physics and mathematics, it can be quite difficult to decide about existence and uniqueness questions for KMS states for a given $C^*$-dynamical system $(\B,\alpha)$ consisting of a $C^*$-algebra $\B$ and dynamics (automorphic $\Rl$-action) $\alpha$, i.e. to determine whether the system allows for thermal equilibrium at a given temperature, or to decide whether several pure thermodynamical phases exist.

This is also true when relevant information about a subsystem is given, i.e. when an $\alpha$-invariant $C^*$-subalgebra $\A\subset\B$ and its KMS states are specified: KMS states of $\A$ might extend uniquely, non-uniquely, or not at all to KMS states of $\B$. 

A special situation of interest arises when the inclusion $\A\subset\B$ is given by a crossed product, i.e. when $\B=\A\rtimes_\gamma\Gamma$ is the crossed product of $\A$ by the action $\gamma$ of a discrete group $\Gamma$ (see \cite{Williams:2007} for a general account of $C^*$-crossed products, and \cite{DoplicherKastlerRobinson:1966,ArakiKastlerTakesakiHaag:1977,ArakiEvans:1983} for some early applications in mathematical physics). 
In this context, the natural question is about the extension of a KMS state~$\om$ from $\A$ to its crossed product $\A\rtimes_\gamma \Gamma$. In this situation, the KMS condition is often of no immediate computational advantage. One rather faces a situation in which the dynamics $\alpha$, the action $\gamma$, the structure of the group $\Gamma$ and the structure of the $C^*$-algebra $\A$ interact in a non-trivial manner.

The nature of the extension problem for KMS states depends in particular on how the dynamics is chosen. In this article, we will always start with a dynamics $\alpha$ on~$\A$ which commutes with the action $\gamma$. Then $\alpha$ extends naturally to $\A\rtimes_\gamma\Gamma$ by acting pointwise; this extended action acts trivially on $\Gamma$. From the point of view of physics, this situation amounts to enlarging the observable algebra $\A$ of a system by elements that are invariant under the dynamics, and ask for the equilibrium states of the enlarged system. Every KMS state on the crossed product then restricts to a KMS state on $\A$.

A different and in some sense opposite choice of dynamics has also been studied in the literature: Instead of starting from a dynamics on $\A$, one defines a dynamics on $\B=\A\rtimes_\gamma\Gamma$ from a 1-cocycle of $\Gamma$, which acts non-trivial on $\Gamma$ but trivial on $\A$. This construction has been investigated in particular in the context of groupoids by Neshveyev \cite{Neshveyev:2014}, and generalized to so-called $\alpha$-regular inclusions $\A\subset\B$ by Christensen and Thomsen \cite{ChristensenThomsen:2020}. As the dynamics is trivial on $\A$, KMS states of $\B$ restrict to (special) traces on $\A$. The two papers mentioned above describe the extension problem for traces on $\A$ to KMS states on $\B$, and provide various conditions under which the extension is unique. In the book of Thomsen \cite[Chapt.~7]{Thomsen:2023}, also a dynamics on $\A\rtimes_\gamma\Gamma$ that combines a dynamics on $\A$ with a cocyle is studied, and a method to construct KMS weights on the crossed product starting from KMS weights on $\A$ is presented. We also refer to  \cite{LacaNeshveyev:2003,Ursu:2021,NeshveyevStammeier:2022,LiZhang:2024} for more related recent work, partially restricted to the case of traces rather than general KMS states.

\medskip 

Motivated by the desire to have a criterion for extensions of KMS states to crossed products (w.r.t. the canonically extended dynamics) that can be efficiently checked in relevant examples, we here propose another point of view on the extension problem. We characterize the KMS states on the crossed product in terms of $\gamma$-twisted KMS functionals as they appear by restriction of a KMS state on the crossed product to the fiber over a fixed group element. 
In this article, we restrict ourselves to crossed products that are as simple as possible as far as the group $\Gamma$ is concerned, namely $\Gamma=\Zl_2$. That is, we consider general (unital) $C^*$-algebras $\A$ with dynamics~$\alpha$ and a commuting involutive automorphism $\gamma$ describing the $\Zl_2$-action, and study the extension problem for KMS states from~$\A$ to $\A\rtimes_\gamma\Zl_2$.

In this case, a $\gamma$-twisted (or graded) KMS functional $\rho:\A\to\Cl$ satisfies by definition
\begin{align}
	\rho(a\alpha_{i\beta}(b))=\rho(b\gamma(a)),\qquad a,b\in\A_\alpha,
\end{align}
in standard notation (see Def.~\ref{def:KMStwisted}). Such ``super'' KMS functionals have been studied in the context of supersymmetry \cite{JaffeLesniewskiWisniowski:1989,Kastler:1989,BuchholzLongo:1999,Hillier:2015}. For us, they play the role of an auxiliary object to describe the untwisted KMS states of the crossed product.

Our general analysis (general $C^*$-algebra $\A$) is contained in Section~\ref{section:abstract}. As we show in Section~\ref{section:twistedKMS}, the extensions of a KMS state $\om$ of $\A$ to the crossed product can be fully characterized in terms of the hermitian $\gamma$-twisted KMS functionals dominated by $\om$ (Theorem~\ref{thm:KMSStateCrossedProduct}). Every KMS state $\om$ on $\A$ has a canonical extension to the crossed product (given by the canonical conditional expectation), but this extension may not be unique. In our approach, the twisted functionals $\rho$ encode in how many different ways a KMS state $\om$ on $\A$ extends to a KMS state on $\A\rtimes_\gamma\Zl_2$, with the choice $\rho=0$ corresponding to the canonical extension. 

In Section~\ref{subsec:GNS_Abstract} we focus on the grading $\gamma$ in the GNS representation $\pi$ of $\om$ and the unitary $V$ implementing it. We show that the selfadjoint elements of the unit ball of the twisted center
\begin{align}
	\Zen(\M,V)=\{x\in\M\,:\,Jx^*J=xV\}
\end{align}
of $\M:=\pi(\A)''$ are in bijection with the extensions of $\om$ (Theorem~\ref{thm:TwistedKMSFunctional}). There are essentially two cases to consider: If $\gamma$ is weakly inner (i.e. inner on $\M$), the extensions of $\om$ are labeled by the selfadjoint elements in the unit ball of the center of $\M$. If, on the other hand, $\gamma$ is freely acting (i.e. outer in the factor case) on $\M$, then $\om$ has a unique extension. In general, $\M$ splits into a direct sum of two components corresponding to these cases, as we show via a theorem of Kallman \cite{Kallman:1969} (Theorem~\ref{thm:AbstractGNS_Split}).
As a corollary, we show that an extremal KMS state on $\A$ has at most two extremal extensions to the crossed product (Corollary~\ref{cor:atmost2}). In the tracial case, a related analysis for general discrete groups has recently been carried out by Ursu \cite{Ursu:2021}.

In Section~\ref{section:asymptotic-abelian}, we focus on the dynamics rather than the grading and investigate the case of (graded) asymptotically abelian dynamics which is of interest in many situations in physics. The connection of such dynamics with twisted KMS functionals has previously been studied by Buchholz and Longo \cite{BuchholzLongo:1999}. In line with their work, we show that in the case of (graded) asymptotically abelian dynamics and non-trivial grading, $\om$ has a unique extension (Corollary~\ref{cor:asymptoticallyabelian}).

\medskip

The results described so far provide a satisfactory analysis of the extension problem for KMS states on $\Zl_2$-crossed products in an abstract setting. Complementing this, we show in Section~\ref{section:CAR} how our criterion can be checked efficiently by spectral analysis in concrete situations of interest in physics. To this end, we specialize to a particular type of $C^*$-dynamical system, given by the CAR algebra $\CAR$ over a Hilbert space. We consider dynamics and grading  implemented by Bogoliubov automorphisms, i.e. by a unitary one-parameter group $(e^{itH})_{t\in\Rl}$ on $\Hil$ and a selfadjoint unitary $G$ commuting with it.

In the non-tracial case without zero modes ($\ker H=\{0\}$), we completely determine all twisted KMS functionals and fully understand the KMS states on the crossed products (Section~\ref{section:KMS-CAR}). It turns out that the Gibbs type condition
\begin{align}
 \Tr_{\Hil_{\rm odd}}(e^{-|\beta H_{\rm odd}|})<\infty
\end{align}
(with $\Hil_{\rm odd}$ the eigenspace of $G$ with eigenvalue $-1$, and $H_{\rm odd}=H|_{\Hil_{\rm odd}}$ the $G$-odd part of the Hamiltonian) holds if and only if there is more than one extension of the unique KMS state of $\CAR$ to its crossed product. All extensions are explicitly determined in Theorem~\ref{thm:CrossedProduct_CAR_KMSStates}. We also analyze the simpler tracial case ($H=0$) in which we have a unique tracial extension (Proposition \ref{prop:tracialcase}).

In Section~\ref{section:physics} we explain how our results apply in examples from mathematical physics. In particular, we show that the extended field algebra of the Ising QFT has the structure of a $\Zl_2$-crossed product with a CAR algebra, and has a unique KMS state at each inverse temperature.

\section{KMS states on $\Zl_2$-crossed products of $C^*$-algebras}\label{section:abstract}

\subsection{General setup}\label{section:generalsetup}
In this section we describe the question we are investigating in a general setting. We will work with unital $C^*$-algebras $\A$ throughout, and write $\States(\A)$ for the state space of $\A$.

\begin{definition}
	\label{def:GradedDynamicalSystem}
	A graded $C^*$-dynamical system is a triple $(\A, \alpha, \gamma)$ consisting of a $C^*$-algebra $\A$, a strongly continuous automorphic $\Rl$-action $\alpha$, and an involutive automorphism $\gamma\in\Aut\A$ satisfying
	\begin{equation*}
		\alpha_t \circ \gamma = \gamma \circ \alpha_t, \qquad t \in \Rl.
	\end{equation*}
\end{definition}

When the grading given by $\gamma$ is absent or not relevant, we refer to $(\A,\alpha)$ as a $C^*$-dynamical system. The action $\alpha$ is also called the dynamics of the system, and $\gamma$ is called the grading. As usual, we write $\A_\alpha\subset\A$ for the norm dense ${}^*$-subalgebra of entire analytic elements for $\alpha$.

\begin{definition} \label{def:KMS}
	Let $(\A, \alpha)$ be a $C^*$-dynamical system.	

	 A state $\omega \in \mathcal{S}(\A)$ is called $\alpha$-KMS state at inverse temperature $\beta \in \Rl$, or $(\alpha,\beta)$-KMS state, if 
	\begin{align}
		\om(a\alpha_{i\beta}(b))
		=
		\om(ba),\qquad a,b\in\A_\alpha.
	\end{align}
	We denote by $\mathcal{S}_\beta(\A, \alpha)$ the set of all $(\alpha, \beta)$-KMS states and by $\KMS(\A,\alpha,\gamma)$ the subset of all $\gamma$-invariant (even) $(\alpha,\beta)$-KMS states.
\end{definition}

We note that if the dynamics $\alpha_t={\rm id}_\A$ is trivial and $\beta\in\Rl$ arbitrary, $\KMS(\A,{\rm id}_\A)$ coincides with the set $\T(\A)$ of tracial states on $\A$. Similarly, if $\beta=0$, for any dynamics $\alpha$ we have ${\mathcal S}_0(\A,\alpha)=\T(\A)$. Without loss of generality, we therefore assume from now on without further mentioning $$\beta\neq0.$$

We next describe $\Zl_2$-crossed products to establish our notation. Given a graded $C^*$-dynamical system $(\A,\alpha,\gamma)$, we consider
\begin{align*}
 \bA:=\A\oplus\A
\end{align*}
as a linear space. This space is a unital ${}^*$-algebra with the product, ${}^*$-operation, and unit 
\begin{align}
	(a,b)(a',b')&=(aa'+b\gamma(b'),\,ab'+b\gamma(a')),\\
	(a,b)^*&=(a^*,\gamma(b^*)),\\
	\bunit&=(1,0).
\end{align}
To describe a norm on $\bA$, we map $(a,b)\in\bA$ to 
\begin{equation*}
	\begin{pmatrix}
		a & b \\
		\gamma(b) & \gamma(a) \\
	\end{pmatrix}
	\in M_2(\A)
\end{equation*}
and define $\|(a,b)\|$ as the operator norm of this matrix acting on the Hilbert $\A$-module $\A^2$. Then $\bA$ is a $C^*$-algebra.

Note that we can equivalently view $\bA$ as the algebra of functions $f:\Zl_2\to\A$ with convolution product by associating with $(a,b)\in\bA$ the function $f(0)=a$, $f(1)=b$. Thus $\bA$ coincides with the crossed product 
\begin{align}
 \bA=\A\rtimes_\gamma\Zl_2
\end{align}
of $\A$ by $\Zl_2$ w.r.t. the $\Zl_2$-action given by $\gamma$, and the above constructed norm with the universal norm \cite[Section~2.5]{Williams:2007}. 

Thanks to the assumption that the dynamics commutes with the grading (Def.~\ref{def:GradedDynamicalSystem}), 
\begin{align}
	\balpha_t((a,b))
	:=
	(\alpha_t(a),\alpha_t(b)),\qquad (a,b)\in\bA,
\end{align}
defines a strongly continuous $\Rl$-action on $\bA$ (we omit the easy proof of the strong continuity), with analytic elements $\bA_\balpha=\A_\alpha\oplus\A_\alpha$. 

\bigskip 

We are interested in understanding the KMS states of the $C^*$-dynamical system $(\bA,\balpha)$ and begin with the following simple observation.

\begin{proposition}\label{proposition:canonical}
	Let $(\A,\alpha,\gamma)$ be a graded $C^*$-dynamical system.
	\begin{enumerate}
		\item Any KMS state $\bom\in\KMS(\bA,\balpha)$ of the crossed product restricts to a $\gamma$-invariant KMS state $\om:=\bom|_{\A}\in\KMS(\A,\alpha,\gamma)$ of $\A$. \label{item:Restriction}
		\item For every $\gamma$-invariant KMS state $\om\in\KMS(\A,\alpha,\gamma)$, there exists a KMS state $\bom\in\KMS(\bA,\balpha)$ such that $\bom|_\A=\om$.
	\end{enumerate}	
\end{proposition}
\begin{proof}
	a) The dynamics $\balpha$ restricts to $\alpha$ on the subalgebra $\A\simeq \A\op 0\subset\bA$. Thus $\om:=\bom|_\A\in\KMS(\A,\alpha)$. To see that $\om$ is $\gamma$-invariant, we first note that $(0,1)(a,0)(0,1)=(\gamma(a),0)$, $a\in\A$. As $(0,1)$ is $\balpha$-invariant and squares to the identity, the KMS condition implies for any $a\in\A$
	\begin{align*}
		\om(\gamma(a))
		&=
		\bom(\gamma(a),0)
		=
		\bom((0,1)(a,0)\balpha_{i\beta}(0,1))
		=
		\bom(a,0)
		=
		\om(a).
	\end{align*}
	b) The canonical faithful conditional expectation
	\begin{align}\label{eq:E}
		E:\bA\to\A,\qquad (a,b)\mapsto a,
	\end{align}
	satisfies $E\circ\balpha_t=\alpha_t\circ E$, $t\in\Rl$. We consider the state $\bom:=\om\circ E$ on $\bA$, and entire elements $(a_1,a_2),(b_1,b_2)\in\bA_\balpha$. Using that $\gamma$ and $\alpha_t$ commute and that $\om$ is $\gamma$-invariant, we find
	\begin{align*}
		\bom((a_1,a_2)\balpha_{i\beta}(b_1,b_2))
		&=
		\om(a_1\alpha_{i\beta}(b_1)+a_2\alpha_{i\beta}(\gamma(b_2)))
		\\
		&=
		\om(b_1a_1+b_2\gamma(a_2))
		\\
		&=
		\bom((b_1,b_2)\,(a_1,a_2)).
	\end{align*}
	Thus $\bom\in\KMS(\bA,\balpha)$.
\end{proof}

We therefore see that understanding $\KMS(\bA,\balpha)$ amounts to understanding the extensions of $\gamma$-invariant KMS states $\om\in\KMS(\A,\alpha,\gamma)$ to $\balpha$-KMS states on $\bA$. As just shown, any $\om$ has a {\em canonical extension}, denoted
\begin{align}\label{eq:CanonicalExtension}
	\bomcan := \om\circ E\in\KMS(\bA,\balpha).
\end{align}
In general, this extension is however not unique. We explain this non-uniqueness in an example below. To prepare this, we first need a lemma on concretely represented crossed products.

\begin{lemma}\label{lemma:simple}

	Let $\Hil$ be a Hilbert space, $\A\subset\B(\Hil)$ a simple unital $C^*$-algebra, and $G$ a selfadjoint unitary on $\Hil$ satisfying $G\A G=\A$ and $G\not\in\A$. Then the $C^*$-algebra $C^*(\A,G)$ generated by $\A$ and $G$ is isomorphic to $\A\rtimes_{\Ad_G}\Zl_2$.
\end{lemma}
\begin{proof}
	The map $\pi:\A\rtimes_{\Ad_G}\Zl_2\to C^*(\A,G)$, $\pi(a,b):=a+bG$ is a surjective homomorphism. To show that $\pi$ is injective, note that $(a,b)\in\ker\pi$ implies $a\in{\mathcal I}:=\{a\in\A\,:\,(\exists b\in\A\,:\,a=bG) \}$, which is a closed ${}^*$-ideal in $\A$. By assumption, $G\not\in\A$, which implies $1\not\in{\mathcal I}$. Since $\A$ is simple, this yields ${\mathcal I}=\{0\}$; hence $\pi$ is an isomorphism.
\end{proof}

We now give an example of a situation in which a KMS state has more than one extension to the crossed product.

\begin{lemma}\label{lemma:Gibbs}
	Let $\A\subset\B(\Hil)$ and $G$ form a concrete realization of $\A\rtimes_{\Ad_G}\Zl_2$ as in the previous lemma, and let $U(t)=e^{itH}$ be a strongly continuous unitary one-parameter group on $\Hil$ commuting with $G$ such that $\alpha_t:=(\Ad_{U(t)})|_\A$ is a dynamics on $\A$.

	If the Gibbs condition
	\begin{align}\label{eq:Gibbs}
		\Tr_\Hil(e^{-\beta\Ham})<\infty
	\end{align}
	is satisfied for some $\beta$, and 
	\begin{align}\label{eq:Gnot0}
		\Tr_\Hil(e^{-\beta\Ham}G)\neq0,
	\end{align}
	then the KMS state $\om_\beta\in\KMS(\A,\alpha)$ defined by
	\begin{align}
		\om_\beta(a) := 
		\frac{\Tr_\Hil(e^{-\beta\Ham }\,a)}{\Tr_\Hil(e^{-\beta\Ham})}
	\end{align}
	does not have a unique extension to the crossed product $\bA=\A\rtimes_{\Ad_G}\Zl_2$.
\end{lemma}
\begin{proof}
	It is well known that the trace class operator $e^{-\beta\Ham}$ defines a KMS state $g_\beta$ at inverse temperature $\beta$ of the $C^*$-dynamical system $(\B(\Hil),\Ad_U)$ given by $g_\beta(a)=\frac{\Tr_\Hil(e^{-\beta\Ham }\,a)}{\Tr_\Hil(e^{-\beta\Ham})}$. This state restricts to $\bomGibbsb:=(g_\beta)|_{\bA}\in\KMS(\bA,\balpha)$ and further to $\om_\beta:=(g_\beta)|_{\A}\in\KMS(\A,\alpha)$. Thus $\bomGibbsb$ is an extension of $\om_\beta$ to the crossed product.
	
	In view of \eqref{eq:Gnot0}, we have $\bomGibbsb(G)\neq0$. Since the canonical extension \eqref{eq:CanonicalExtension} satisfies $\bomcanb(G) = \om_\beta(E(0,1))=0$, we have two different extensions $\bomcanb$ and $\bomGibbsb$ of $\om_\beta$.
\end{proof}

\subsection{Characterization of extensions through twisted KMS-functionals}\label{section:twistedKMS}
The canonical extension $\om\mapsto\bomcan$ provides us with an injective mapping $\KMS(\A,\alpha,\gamma)\to\KMS(\bA,\balpha)$ that is not surjective in general. To obtain a bijective characterization of $\bom\in\KMS(\bA,\balpha)$, we also need to include information about $\rho:=\bom|_{0\oplus\A}$ which is not contained in $\bom|_{\A\oplus0}$. Such a characterization will be given below. The main points are that $\rho$ will be hermitian in the sense $\rho^*(a):=\overline{\rho(a^*)}=\rho(a)$ and dominated by $\om:=\bom|_{\A\oplus0}$, and that $\rho$ satisfies a $\gamma$-twisted version of the KMS condition.

\begin{definition}\label{def:KMStwisted}
	Let $(\A,\alpha)$ be a $C^*$-dynamical system and $\alpha\in\Aut\A$. A functional $\rho$ on $\A$ is called $\gamma$-twisted $\alpha$-KMS functional at inverse temperature $\beta \in \Rl$, or $(\alpha,\beta,\gamma)$-KMS functional, if 
	\begin{align}
		\rho(a\alpha_{i\beta}(b))
		=
		\rho(b\gamma(a)),\qquad a,b\in\A_\alpha.
	\end{align}
	We denote by $\mathcal{F}_\beta(\A, \alpha, \gamma)$ the set of all $\gamma$-twisted $(\alpha, \beta)$-KMS functionals.

	Given a state $\omega \in \mathcal{S}(\A)$, we denote by
	\begin{equation*}
		\mathcal{F}_{\beta}(\A, \alpha, \gamma, \omega) := \{\rho \in \mathcal{F}_\beta(\A, \alpha, \gamma) \: | \: \rho = \rho^* , \,\abs{\rho(a^*b)}^2 \leq \omega(a^*a) \omega(b^*b) \}
	\end{equation*}
	the set of hermitian $\gamma$-twisted KMS functionals dominated by $\omega$.
\end{definition}

Our main interest lies in KMS states, and twisted KMS functionals will only appear as a tool to characterize KMS states on crossed products. We refer to \cite{JaffeLesniewskiWisniowski:1989,Kastler:1989,BuchholzLongo:1999,Hillier:2015} for works more directly concerned with twisted KMS functionals.

Note that the functionals $\rho\in\mathcal{F}_{\beta}(\A, \alpha, \gamma, \omega)$ are typically not positive \cite{BuchholzLongo:1999}, but always continuous because we have
\begin{align}
	|\rho(b)| \leq \sqrt{\om(b^*b)} \leq \|b\|,\qquad b\in\A,\quad\rho\in\mathcal{F}_{\beta}(\A, \alpha, \gamma, \omega)
\end{align}
by the domination property $\abs{\rho(a^*b)}^2 \leq \omega(a^*a) \omega(b^*b)$. This property is equivalent to $\om-\rho$ and $\om+\rho$ being positive, but we will not need this characterization. We also recall that $\gamma$-twisted KMS functionals on a unital graded $C^*$-dynamical system $(\A,\alpha,\gamma)$ are automatically $\gamma$-invariant and $\alpha$-invariant, see \cite{BuchholzLongo:1999}.

With these definitions, the KMS states of the crossed product $\bA$ can be characterized in terms of KMS states and twisted KMS functionals of $\A$.

\begin{theorem} \label{thm:KMSStateCrossedProduct}
Let $(\A, \alpha, \gamma)$ be a graded $C^*$-dynamical system and let $\bA$ be its $\Zl_2$-crossed product. There is a bijection
\begin{equation*}
	\KMS(\bA,\balpha)
	\simeq
	\{ (\omega, \rho) \, | \, \omega \in \mathcal{S}_\beta(\A, \alpha, \gamma), \rho \in \mathcal{F}_{\beta}(\A, \alpha, \gamma, \omega) \}
\end{equation*}
carrying a KMS state $\bom\in\KMS(\bA,\balpha)$ into the pair $(\omega, \rho)$ given by
\begin{align}
	\om=\bom|_{\A\oplus0},\quad \rho=\bom|_{0\oplus\A},
\end{align}
and a pair $(\om,\rho)$ into the KMS state
\begin{align}\label{eq:extension}
	\bom(a,b)=\om(a)+\rho(b),\qquad a,b\in\A.
\end{align}
\end{theorem}

Before giving the proof, we remark that the sets $\mathcal{S}_{\beta}(\A, \alpha, \gamma)$ and $\mathcal{F}_{\beta}(\A, \alpha, \gamma, \omega)$ are convex and that if $\rho \in \mathcal{F}_{\beta}(\A, \alpha, \gamma, \omega)$, so is the line segment $[- \rho, \rho] = \{\lambda \rho \, | \, \lambda \in [-1,1]\}$. Especially $0 \in \mathcal{F}_{\beta}(\A, \alpha, \gamma, \omega)$ for every KMS state $\omega$, which recovers the canonical extension $\bomcan\cong(\om,0)$. Thus $\mathcal{F}_{\beta}(\A, \alpha, \gamma, \omega)$ is in bijection with all KMS extensions of $\om\in\KMS(\A,\alpha,\gamma)$ to $\bA$.

We also remark that in case $\om\in\mathcal{S}_{\beta}(\A, \alpha, \gamma)$ admits a non-canonical extension $\bom=(\om,\rho)$ given by a twisted functional $\rho\neq0$, $\bomcan$ is not an extremal KMS state because it can be realized as the convex combination
\begin{align}
	\bomcan=\frac12(\om,\rho)+\frac12(\om,-\rho).
\end{align}

\begin{proof}
Take a KMS state $\omega \in \mathcal{S}_{\beta}(\A, \alpha, \gamma)$ and a twisted KMS functional \linebreak $\rho \in \mathcal{F}_{\beta}(\A, \alpha, \gamma, \omega)$. Then $\bom$ \eqref{eq:extension} is clearly linear and $\bom(\bunit)=1$. To show positivity of $\bom$, let $a,b\in\A$. Using the $\gamma$-invariance of $\om$ and $\rho$ and the fact that $\om$ dominates $\rho$, we find
\begin{align}\label{eq:KMSpos}
	\bom((a,b)^*(a,b))
	&=
	\om(a^*a+\gamma(b^*b))+\rho(a^*b+\gamma(b^*a))
	\\
	&\geq \omega(a^*a) + \omega(b^*b)-
	|\rho(a^*b)|-|\rho(b^*a)|
	\nonumber
	\\
	&\geq \omega(a^*a) + \omega(b^*b)
	-
	2\,\omega(a^*a)^{\frac12}\omega(b^*b)^\frac{1}{2}
	\nonumber
	\\
	&=
	\left(\omega(a^*a)^\frac{1}{2}-\omega(b^*b)^\frac{1}{2}\right)^2
	\geq0.
	\nonumber
\end{align}
Therefore $\bom$ is a state on $\bA$.

To verify the KMS condition, let $(a,b),(a',b')\in\bA_{\balpha}=\A_\alpha\oplus\A_\alpha$ be entire analytic for $\balpha$. Then,
\begin{align}\label{eq:KMScalc}
	\bom((a',b')\balpha_{i\beta}(a,b))
	&=
	\om(a'\alpha_{i\beta}(a)+b'\gamma(\alpha_{i\beta}(b)))
	+
	\rho(a'\alpha_{i\beta}(b)+b'\gamma(\alpha_{i\beta}(a)))
	\nonumber
	\\
	&=
	\om(a'\alpha_{i\beta}(a))+\om(\gamma(b')\alpha_{i\beta}(b))
	+
	\rho(a'\alpha_{i\beta}(b))+\rho(\gamma(b')\alpha_{i\beta}(a))
	\nonumber
	\\
	&=
	\om(a a')+\om(b\gamma(b'))
	+
	\rho(b\gamma(a'))+\rho(a b')
	\\
	&=
	\bom((a,b)(a',b')),
	\nonumber
\end{align}
so $\bom$ is a KMS state for $\balpha$.

Conversely, let $\bom \in \KMS(\bA,\balpha)$ be a KMS state. By Lemma~\ref{proposition:canonical}~\ref{item:Restriction}, its restriction $\om:=\bom|_{\A \op 0}$ is a $\gamma$-invariant KMS state for the restricted dynamics $\alpha$.

It is clear that $\rho:=\bom|_{0\oplus\A}$ is a linear functional on $\A$. A look at calculation \eqref{eq:KMScalc} with $a=b'=0$ yields $\rho(a'\alpha_{i\beta}(b))=\rho(b\gamma(a'))$ as a consequence of the KMS property of $\bom$. Thus $\rho$ is a $\gamma$-twisted KMS functional on $\A$.

It remains to show that $\rho$ is hermitian and dominated by $\om$. This is a consequence of the positivity of $\bom$. Indeed, the first line of \eqref{eq:KMSpos} is positive by positivity of $\bom$. As $a,b\in\A$ are arbitrary, this implies that the matrix
\begin{equation*}
M=
\begin{pmatrix}
\omega(a^*a) & \rho(a^*b) \\
\rho(b^*a ) & \omega(b^*b) \\
\end{pmatrix}
\geq0
\end{equation*}
is positive semi-definite. In particular, it is selfadjoint, implying that $\rho$ is hermitian, and
\begin{equation*}
0\leq\det(M) = \omega(a^*a)\omega(b^*b) - \rho(a^*b)\rho(b^*a )
= \omega(a^*a)\omega(b^*b) - |\rho(a^*b)|^2,
\end{equation*}
shows that $\rho$ is dominated by $\omega$.

The operations relating $\bom$ to $(\om,\rho)$ and vice versa are clearly inverses of each other, concluding the proof.
\end{proof}

\medskip 

We mention in passing that various aspects of this analysis generalize to crossed products with more general groups $\Gamma$ as long as a conditional expectation $\A\rtimes\Gamma\to\A$ exists; this is in particular the case for any discrete group. We leave a detailed investigation of such a generalization to a future work.

\subsection{Twisted KMS functionals and the twisted center}
\label{subsec:GNS_Abstract}

In this section, we characterize the set of twisted KMS functionals dominated by a given KMS state of a $C^*$-algebra $\A$ in terms of their associated noncommutative Radon-Nikod\'ym derivatives. 

Given a graded $C^*$-dynamical system $(\A, \alpha, \gamma)$ with a $\gamma$-invariant KMS state~$\omega$, we consider the associated GNS triple $(\Hil, \pi, \Omega)$. It is well known that $\Omega$ is a cyclic and separating vector for the enveloping von Neumann algebra $\M := \pi(\mathcal{A})'' \subset \B(\Hil)$ \cite{BratteliRobinson:1997}. The corresponding modular objects are denoted by $(J,\Delta)$. The $\Zlt$-grading of $\A$ defines a unique unitary $V \in \B(\Hil)$ satisfying
\begin{equation}
V \pi(a) V = \pi(\gamma(a)), \quad V \Omega = \Omega \quad \text{and} \quad V^2 = 1_{\Hil},
\end{equation}
which extends the grading to $\M$. The modular group $(\Delta^{it})_{t \in \Rl}$ implements the (rescaled) dynamics $\alpha$ on the GNS space via
\begin{equation}
	\Delta^{it} \pi(a) \Omega = \pi(\alpha_{-t \beta}(a)) \Omega.
\end{equation}
As $\gamma$ commutes with $\alpha$, the associated unitary $V$ commutes with both the modular conjugation and group. We now shift the focus to the von Neumann algebra $\M$ and introduce some notation.

\begin{definition}\label{def:TwistedCenter}
	Let $\M \subset \B(\Hil)$ be a von Neumann algebra with cyclic and separating vector $\Omega$ and $(J,\Delta)$ the associated modular data. 
	
	A $\Zlt$-grading of $(\M, \Om)$ is an operator $V \in \B(\Hil)$ s.t.
	\begin{equation}
	V= V^* = V^{-1}, \quad \Ad_V : \M \to \M \qand V \Omega = \Omega.
	\end{equation}
	We then call $(\M, \Omega, V)$ a graded $W^*$-dynamical system or simply graded von Neumann algebra. Further, we define the $V$-twisted center of $\M$ by
	\begin{equation}\label{eq:ZenV}
	\Zen(\M, V) := \{x \in \M \, | \, Jx^*J = xV \}.
	\end{equation}
\end{definition}

A small calculation shows that the $\Zlt$-grading $V$ and the Tomita operator $S$ commute on $\M \Om$, so 
$V$ commutes with both the modular conjugation $J$ and the modular group $(\Delta^{it})_{t \in \Rl}$.

Given a graded von Neumann algebra $(\M, \Om, V)$, we denote the center of $\M$ by $\Zen(\M)$ and recall that $\Zen(\M) = \{z \in \M \, | \, Jz^*J = z \}$ which motivates the term $V$-twisted center for \eqref{eq:ZenV}. As usual, $\M^\alpha$ denotes the elements invariant under the modular group, $\M^\gamma$ denotes the $\Ad_V$-invariant elements and correspondingly $\Zen(\M)^\gamma$ the $\Ad_V$-invariant central elements. Furthermore, the centralizer of $(\M, \Om)$ is denoted by $\Zen_\Omega(\M)$ which coincides with $\M^\alpha$ by modular theory, see \cite[Proposition 5.3.28]{BratteliRobinson:1997}.

We now show that the set $\mathcal{F}_{\beta}(\A, \alpha, \gamma, \omega)$ which characterizes the extensions of the KMS state $\omega$ to the crossed product $\bA$ can be identified with the selfadjoint elements of $\Zen(\M, V)$ lying in the unit ball $B_\M$.

\begin{theorem} \label{thm:TwistedKMSFunctional}
	Let $(\A, \alpha, \gamma)$ be a graded C$^*$-dynamical system and \linebreak $\omega \in \mathcal{S}_{\beta}(\A, \alpha, \gamma)$ and $(\M, \Om, V)$ the enveloping graded $W^*$-dynamical system.
	
	Then there exists a bijection
	\begin{align*}
		\Xi:\mathcal{F}_{\beta}(\A, \alpha, \gamma, \om) &\to B_\M\cap \Zen(\M, V)_{s.a.}\\
		 \rho&\mapsto R_\rho
	\end{align*}
	preserving convex combinations, where $R_\rho$ is defined by $ \rho(a) := \scp{\Om}{\pi(a) R_\rho \Om}$.
\end{theorem}

\begin{proof}
	Let $\rho \in \mathcal{F}_{\beta}(\A, \alpha, \gamma, \omega)$ and consider the densely defined sesquilinear form 
	\begin{equation*}
	h_\rho: \pi(\A) \Omega \times \pi(\A) \Omega \to \Cl, \quad (\pi(a)\Omega, \pi(b) \Omega) \mapsto \rho(a^*b)
	\end{equation*}
	on $\Hil$ and note that $h_\rho$ is well-defined since $\rho$ is dominated by $\om$. The sesquilinear form $h_\rho$ is bounded by $1$, since for all $a,b \in \A$,
	\begin{equation*}
		\abs{h_\rho(\pi(a)\Omega, \pi(b)\Omega)} = \rho(a^* b) \leq (\omega(a^*a)\omega(b^*b))^{\frac{1}{2}} = \norm{\pi(a)\Omega}\norm{\pi(b)\Omega}.
	\end{equation*}
	Therefore, $h_\rho$ can be uniquely extended to a continuous sesquilinear form on $\Hil$ also denoted by $h_\rho$. Therefore, there exists a unique operator $R_\rho' \in \B(\Hil)$ satisfying
	\begin{equation*}
	\rho(a^*b) = h_\rho(\pi(a)\Omega,\pi(b)\Omega) = \scp{\pi(a)\Omega}{R_\rho'\pi(b)\Omega},\ \quad \forall a, b \in \A .
	\end{equation*} 

	It follows from the boundedness of $h_\rho$ that $\norm{R_\rho'} \leq 1$. Moreover, $R_\rho' \in \pi(\A)' = \M'$  because
	\begin{align*}
		\scp{\pi(a)\Omega}{\comm{R_\rho'}{\pi(c)}\pi(b)\Omega} = \rho(a^*cb) - \rho((c^*a)^*b) = 0.
	\end{align*}
	It is clear that $\rho$ is hermitian if and only if the operator $R_\rho'$ is selfadjoint.
	
	We rewrite, for $a, b \in \A_\alpha$, both sides of the twisted KMS condition.
	\begin{align*}
		\rho(a^*\alpha_{i\beta}(b)) &= \rho(\alpha_{\frac{-i\beta}{2}}(a^*)\alpha_{\frac{i\beta}{2}}(b)) = \scp{\Delta^\frac{1}{2} \pi(a) \Omega}{R_\rho' \Delta^\frac{1}{2} \pi(b) \Omega} \\
		\rho(b \gamma(a^*)) &= \scp{\pi(b^*) \Om}{R_\rho' V\pi(a^*) \Om} = \scp{ \Delta^\frac{1}{2} \pi(a) \Om}{ J R_\rho' V J\Delta^\frac{1}{2} \pi(b)\Om}
	\end{align*}
	As $\Delta^\frac{1}{2}\pi(\A_\alpha) \Omega$ is dense in $\Hil$, it follows that $R_\rho' = JR_\rho'VJ$. Rewriting this result in terms of $R_\rho := JR_\rho'J \in \M$ yields
	\begin{equation*}
		 JR_\rho J=R_\rho V.
	\end{equation*}
	The defining equation of $R_\rho'$ can be rewritten as
	\begin{equation*}
		\rho(a) = \scp{\Om}{R_\rho'\pi(a) \Om} = \scp{\Om}{\pi(a) JR_\rho J \Om} = \scp{\Om}{\pi(a) R_\rho V \Om} = \scp{\Om}{\pi(a) R_\rho \Om},
	\end{equation*}
	proving that $\Xi$ is well-defined and injective.
	
	The map $\Xi$ clearly preserves convex combinations. Furthermore, by analogous arguments, an operator $R\in B_\M\cap \Zen(\M, V)_{s.a.}$ defines a unique $\om$-dominated $(\alpha, \beta, \gamma)$-KMS functional $\rho_R$ by $\rho_R(a):=\scp{\Om}{\pi(a)R\Om}$.
\end{proof}

Notice that $\mathcal{F}_{\beta}(\A, \alpha, \gamma, \om)$ is compact in the weak-$^*$ topology, $B_\M\cap \Zen(\M, V)_{s.a.}$ is compact in the weak-operator topology, and $\Xi$ is continuous with respect to these topologies. As a consequence of Krein-Milman theorem, both sets are the closed convex hull of their respective extreme points. Furthermore, since $\Xi$ preserves convex combinations, it also gives a bijection between the extreme points of these two sets.

Theorem \ref{thm:TwistedKMSFunctional} now motivates us to study the $V$-twisted center of $\M$. The results for $\Zen(\M, V)$ can then be applied to the enveloping von Neumann algebra $\M$ for a KMS state $\om$ on a $C^*$-algebra $\A$. We collect some properties of the $V$-twisted center here.

\begin{lemma} \label{lem:TwistedCenterProperties}
	Let $(\M, V, \Omega)$ be a graded von Neumann algebra. Then the following statements hold:
	\begin{enumerate}
		\item $\Zen(\M, V)$ is a $^*$-invariant WOT-closed complex subspace of $\M$;\label{item:Subspace}
		\item $\Zen(\M, V) \subset \M^\gamma \cap \M^\alpha$; \label{item:inclusion}
		\item If $x\in \Zen(\M,V)$, then \label{item:free}
		\begin{align}\label{eq:ZenVinner}
			xy=\Ad_V(y)x,\qquad y\in \M;
		\end{align}
		\item $x_1x_2 \in \Zen(\M)^\gamma$ for all $x_1,x_2 \in \Zen(\M, V)$; \label{item:prodtwo}
		\item $xz = zx \in \Zen(\M, V)$ for all $x \in \Zen(\M, V)$ and $z \in \Zen(\M)$; \label{item:prodcentral}
		\item For $x \in \Zen(\M, V)$ with polar decomposition $x = u\abs{x}$ it follows that ${u \in \Zen(\M, V)}$ and $\abs{x} \in \Zen(\M)^\gamma$. \label{item:PolarDecomposition}
	\end{enumerate}

	In particular, if $\Ad_V$ is an inner action on $\M$, i.e. there exists unitary $u \in \M$ with $\Ad_V = \Ad_u$ on $\M$, then
	\begin{equation*}
		\Zen(\M, V) = u \cdot \Zen(\M).
	\end{equation*}
\end{lemma}

\begin{proof}
	a) Clearly, $\Zen(\M, V)$ is a $^*$-invariant complex subspace of $\M$ and the defining relation depends WOT-continuously on $x \in \Zen(\M, V)$.
	
	b) Take $x \in \Zen(\M, V)$, then 
	\begin{equation*}
	VxV = VJx^*J = J Vx^*J = J(xV)^*J = J(Jx^*J)^*J = x,
	\end{equation*}
	proving $\Zen(\M, V) \subset \M^\gamma$. The invariance under the modular group follows by first considering $J \Delta^{\frac{1}{2}} x \Omega = S x \Omega = x^* \Omega = JxJV \Omega = J x \Omega,$
	where $S$ is the Tomita operator. Thus $x \Omega$ is eigenvector of $\Delta^\frac{1}{2}$ to the eigenvalue $1$. Hence, it is invariant under the modular group and $x \Om = \Delta^{it} x \Om = \Delta^{it} x \Delta^{-it} \Om$.
	As $\Om$ is separating for $\M$, it follows that $x = \Delta^{it} x \Delta^{-it}$ and $x \in \M^\alpha$.
	
	c) Let $x \in \Zen(\M, V)$. For $y \in \M$
	\begin{equation*}
	xy = (Jx^*J)Vy = (Jx^*J)\Ad_V(y)V = \Ad_V(y)(Jx^*J)V = \Ad_V(y)x.
	\end{equation*}
	
	d) From item \ref{item:free},  $xy=\Ad_V(y)x$ for all $x\in \Zen(\M,V)$ and $y\in \M$. As $\Ad_V \circ \Ad_V = \Ad_{V^2} = \id$ and $\Zen(\M, V) \subset \M^\gamma$, it follows that $x_1x_2$ is central and even for $x_1,x_2 \in \Zen(\M, V)$. 
	
	e) Take $x \in \Zen(\M, V)$ and $z \in \Zen(\M)$ and consider
	\begin{equation*}
	JxzJ = JxJJzJ = x^*V z^* = x^*\Ad_V(z^*)V = (\Ad_V(z)x)^*V = (xz)^*V.
	\end{equation*}
	
	f) Take $x \in \Zen(\M, V)$ with polar decomposition $x = u\abs{x}$. Then, $\abs{x} = \sqrt{x^*x} \in \Zen(\M)^\gamma$, since $x^*x \in \Zen(\M)^\gamma$ as a consequence of \ref{item:prodtwo} and $\Zen(\M)^\gamma$ is a weakly closed subalgebra of $\M$.	The partial isometry $u$ satisfies $u\abs{x} = x =Jx^*VJ = Ju^*VJ \abs{x}$. By the uniqueness of the polar decomposition and
	\begin{equation*}
	\ker(Ju^*VJ) = JV\ker(u^*) = JV\ker(x^*) = \ker(Jx^*VJ) = \ker(x) = \ker(u),
	\end{equation*}	
	one concludes $Ju^*J = uV$.
	
	Assume now there exists a unitary $u \in \M$ with $\Ad_V = \Ad_u$ on $\M$. Then $u \in \M^\gamma$ and furthermore $u \in \M^\alpha = \Zen_\Om(\M)$ because
	\begin{equation*}
		\scp{\Om}{u y \Om} = \scp{\Om}{u(y u) u^* \Om} = \scp{\Om}{\Ad_V(y u) \Om} = \scp{\Om}{yu \Om },\qquad y\in\M.
	\end{equation*}
	Then $uV \in \M'$ and $\Om$ separating implies with $Ju^*J \Om = u \Om = u V \Om$ that $u \in \Zen(\M,V)$. Moreover, for $x \in \Zen(\M,V)$ we have $x = u(u^*x) \in u \cdot \Zen(\M)$ by item \ref{item:prodtwo}. The opposite inclusion follows from item \ref{item:prodcentral}. 
\end{proof}

Note that by Lemma~\ref{lem:TwistedCenterProperties}~\ref{item:PolarDecomposition} every nonzero $x \in \Zen(\M, V)$ gives rise to a KMS state $\nu_x$ on $\M$ by
\begin{equation*}
	\nu_x(y) := \norm{\abs{x}^\frac{1}{2} \Om}^{-2} \scp{ \abs{x}^\frac{1}{2} \Om}{y \abs{x}^\frac{1}{2} \Om},
\end{equation*}
due to the correspondence between normal KMS states on $\M$ and the operators affiliated to $\Zen(\M)$, see \cite[Proposition~5.3.29]{BratteliRobinson:1997}. This coincides with the abstract polar decomposition of twisted KMS functionals. Every nonzero twisted KMS functional $\rho$ defines a KMS state $\abs{\rho}$ by polar decomposition of functionals and subsequent normalization, the resulting grading $\Ad_V$ in the GNS space of $\abs{\rho}$ is then given by a inner unitary, see \cite{BuchholzLongo:1999}. 

Conversely, we have seen above that a weakly inner grading $\gamma$ on $\A$ leads to the non-trivial twisted center $\Zen(\M,V)=u \cdot \Zen(\M)$ with the unitary $u\in\M$ implementing $\gamma$, and hence to a complete clarification of the extension problem. This of course includes the special case of a grading that is inner on the $C^*$-algebra $\A$.

In the interest of a criterion that can be checked in concrete applications, we also note the following more concrete situation.

\begin{proposition}\label{prop:weaklyinnergrading}
	Let $(\A,\alpha,\gamma)$ be a graded $C^*$-dynamical system and $\om\in\KMS(\A,\alpha,\gamma)$ with associated graded von Neumann algebra $(\M, \Om, V)$. Suppose $\A$ contains a sequence $(u_n)_{n\in\Nl}$ of unitary selfadjoint elements that are invariant under the dynamics and grading and further satisfy
	\begin{align}
	\lim_{n,m\to\infty}\om(u_nu_m)&=1,\qquad
	\lim_{n\to\infty}\|u_nau_n-\gamma(a)\|=0,\qquad a\in \A_0,
	\end{align}
	where $\A_0\subset\A_\alpha$ is a norm dense ${}^*$-subalgebra invariant under the dynamics. Then	\begin{align}
		\rho(a):=\lim_{n\to\infty}\om(au_n),\qquad a\in\A,
	\end{align}
	is a hermitian $\gamma$-twisted KMS functional dominated by $\om$, and the family of all such functionals is $\{\rho(\,\cdot\, z)\,:\,z=z^*\in\Zen(\M),\;\|z\|\leq1\}$.
\end{proposition}
\begin{proof}
	We first show the existence of the limit defining $\rho$. For $a\in\A$, we estimate
	\begin{align*}
	|\om(a(u_n-u_m))|^2
	&\leq
	\om(a^*a)\om(2-u_nu_m-u_mu_n)
	=
	2\om(a^*a)\cdot(1-\om(u_nu_m)),
	\end{align*}
	where we have used the $\alpha$-invariance of the $u_n$ and the KMS property of $\om$ in the last step. As $\om(u_nu_m)\to1$, this shows that $\rho$ exists as a functional on $\A$.
	
	Similarly, for $a\in\A,b\in\A_0$, we have 
	\begin{align*}
	\om(au_nb)
	&=
	\om(a[u_nbu_n-\gamma(b)]u_n)+\om(a\gamma(b)u_n) \to \rho(a\gamma(b))
	\end{align*}
	because $\|u_nbu_n-\gamma(b)\|\to0$. By uniform boundedness, this implies the existence of the WOT-limit $R:=\lim_n\pi(u_n)=R^*$. But as $\|R\Om\|^2=\lim_{n,m}\om(u_nu_m)=1$, we also have $\|\pi(u_n)\Om-R\Om\|^2\to1-\|R\Om\|^2=0$ and hence $\|(\pi(u_n)-R)x'\Om\|\to0$ for all $x'\in\M'$. As $\Om$ is cyclic for $\M'$, we arrive at $\pi(u_n)\to R$ in SOT, and hence at the unitarity of $R$. It is then clear from our assumptions that $RxR^*=VxV$ for all $x\in\M$. The claims now follow from Lemma~\ref{lem:TwistedCenterProperties}.
\end{proof}

Having clarified the case of weakly inner gradings, we now consider the opposite extreme. The following notion of outerness of an automorphism was introduced in \cite{Kallman:1969}. 

\begin{definition}
	A $^*$-automorphism $\gamma$ acting on a von Neumann algebra $\M$ is called freely acting if for $x \in \M$ the relation $xy = \gamma(y)x, \,y \in \M$ implies $x = 0$.
\end{definition}

In \cite[Theorem~1.11]{Kallman:1969} it is shown by a Zorn argument that every \linebreak $^*$-automorphism of a von Neumann algebra can be uniquely decomposed into an inner and a freely acting part through a central projection. We use this result to prove a similar decomposition of the $V$-twisted center by a central projection. For a central $\gamma$-invariant projection $p$ we denote $\M_p = p\M p$, $V_p = pVp$.

\begin{theorem} \label{thm:AbstractGNS_Split}
	Let $(\M, \Om, V)$ be a graded von Neumann algebra. Then there exist a unique maximal central projection $p \in \Zen(\M)^\gamma$, with the property that there exists a partial isometry $u_p \in \Zen(\M_p, V_p)$ with $u_p^*u_p = u_pu_p^*= p$. Furthermore, $V_p = u_pJu_pJ$ and
	\begin{equation*}
	\Zen(\M, V) = \Zen(\M_p, V_p) \op \Zen(\M_{p^\perp}, V_{p^\perp}) = u_p \cdot \Zen(\M_p) \op \{0\}.
	\end{equation*}
\end{theorem}

\begin{proof}
	By \cite[Theorem~1.11]{Kallman:1969} $\M$ can be uniquely decomposed into $\M = \M_1 \op \M_2$ such that $\Ad_V = \phi_1 \op \phi_2$ where $\phi_i(\M_i) \subset \M_i$, $\phi_1$ is inner on $\M_1$ and $\phi_2$ is freely acting on $\M_2$. We denote the central projection associated to this composition by $p = 1 \op 0$, which is clearly $\Ad_V$-invariant by $\Ad_V(p) = \phi_1(1) \op 0 = p$. Then $\phi_1 \op \phi_2 = \Ad_{V_p} \op \Ad_{V_{p^\perp}}$ by
	\begin{equation*}
		(\phi_1 \op \phi_2)(x \op y) = \Ad_V(x \op y) = \Ad_V(x) \op \Ad_V(y) = (\Ad_{V_p} \op \Ad_{V_{p^\perp}})(x \op y).
	\end{equation*}
	The cyclic and separating vector $\Om$ similarly decomposes into $\Om = p\Om \op p^\perp \Om \in p\Hil \op p^\perp \Hil$, where $\Hil$ is the Hilbert space $\M$ acts on. Using the centrality of $p$, a straightforward calculation shows that $p \Om$ is cyclic and separating for $\M_1 = \M_p$ and similarly $p^\perp$ is cyclic and separating for $\M_2 = \M_{p^\perp}$. Therefore $(\M,\Om,V)$ decomposes as a graded von Neumann algebra and so does
	\begin{equation*}
		\Zen(\M,V) = \Zen(\M_p, V_p) \op \Zen(\M_{p^\perp}, V_{p^\perp}).
	\end{equation*}
	
	The automorphism $\phi_2 = \Ad_{V_{p^\perp}}$ is freely acting on $\M_2$, therefore Lemma \ref{lem:TwistedCenterProperties}~\ref{item:free} directly shows $\Zen(\M_{p^\perp}, V_{p^\perp}) = \{0\}$. As $\phi_1 = \Ad_{V_p}$ is inner on $\M_p$, there exists a $u_p \in \M_p$ with $\Ad_{u_p} = \Ad_{V_p}$ and $u_p^* u_p = u_pu_p^*= p$. Lemma \ref{lem:TwistedCenterProperties} then shows
	\begin{equation*}
		\Zen(\M_p, V_p) = u_p \cdot \Zen(\M_p).
	\end{equation*}
	As $u_p$ is unitary on $p\Hil$, the equation $u_p V_p = Ju_p^*J$ can be rewritten in the form $V_p = u_pJu_pJ$.
	
	We now show that $p$ is the unique maximal projection with the assumed property. This follows from the uniqueness of the decomposition of $\Ad_V$ into an inner and freely acting part in \cite[Theorem~1.11]{Kallman:1969}. Assume there exists another central projection $q$ with a partial isometry $u_q \in \Zen(\M_q, V_q)$ satisfying $u_q^* u_q = q$. Then $\Ad_{V_q}$ is an inner automorphism of $\M_q$ and therefore $\M_q \subset \M_p$ and $q \leq p$.
\end{proof}

We now discuss some direct consequences of Theorem \ref{thm:AbstractGNS_Split}.

\begin{corollary} \label{cor:FreelyActing}
	Let $(\M, \Om, V)$ be a graded von Neumann algebra. Then, $\Zen(\M,V) = \{ 0 \}$ if and only if $\Ad_V$ is freely acting on $\M$.
\end{corollary}

A similar result has recently been derived for traces on crossed products by discrete groups, see \cite{Ursu:2021}. Here we emphasize that the enveloping von Neumann algebra $(\M, \Om)$ of a $C^*$-dynamical system $(\A, \alpha)$ with KMS state $\om$ depends on the dynamics $\alpha$. Therefore, whether $\Ad_V$ is freely acting on $\M$ depends on $\alpha$.

\begin{corollary}\label{cor:one-dimensional}
	Let $(\M, \Om, V)$ be a graded von Neumann algebra and $\M$ a factor. Then $\Zen(\M, V)$ is at most one-dimensional.
\end{corollary}

\begin{proof}
	The only central projections in a factor are $0$ and $1$. The case $\Zen(\M, V)=\{0\}$ is trivial. Therefore, suppose now $\Zen(\M, V)\neq\{0\}$. Then, Theorem \ref{thm:AbstractGNS_Split} yields a unitary $u \in \Zen(\M,V)$ and $\Zen(\M,V) = u \cdot \Cl$.
\end{proof}

This corollary can now be applied to the enveloping von Neumann algebra $\M$ corresponding to an extremal KMS state $\om$ on $\A$. 

\begin{corollary}\label{cor:atmost2}
	Let $(\A,\alpha,\gamma)$ be a graded $C^*$-dynamical system and $\om$ extremal in $\mathcal{S}_\beta(\A, \alpha, \gamma)$. Then the set $\mathcal{F}_{\beta}(\A, \alpha, \gamma, \omega)$ has at most two extremal points.
\end{corollary}

This corollary in combination with Theorem \ref{thm:KMSStateCrossedProduct} implies that there exist at most two extremal extensions of a given extremal $\gamma$-invariant KMS state $\om$ on $\A$ to a KMS state $\bom$ on $\bA$.

It is clear that this is no longer true for non-extremal $\om$. To give an explicite example, let $(\M, \Om, V)$ be a graded von Neumann algebra with $\M$ a factor and $u \in \Zen(\M,V)$ unitary. $u$ can then be chosen selfadjoint as well, because $u^2 \in \Zen(\M) = \Cl \cdot 1$. Define $\hat{\M}=\M\ot D_n$, where $D_n$ denotes the set of $n\times n$ diagonal matrices, and $\hat{V}=V\ot \id_n$. Then $\hat{u}:=u\ot \id_n\in $ is invertible and $\Zen(\hat\M, \hat V)_{s.a.}= \hat{u} \cdot \Zen(\hat \M)_{s.a} =(\Rl u) \ot (D_n)_{s.a.}$. Therefore, $B_{\hat\M}\cap \Zen(\hat\M,\hat V)_{s.a.}$ has $2^n$ extreme points.

\subsection{Twisted KMS functionals and asymptotically abelian dynamics}\label{section:asymptotic-abelian}

We now shift the focus from the grading $\gamma$ to the dynamics $\alpha$ and discuss weak asymptotic commutativity, as it plays a prominent role in many physical models \cite{BratteliRobinson:1997, BuchholzLongo:1999}.

\begin{definition}
	A graded $C^*$-dynamical system $(\A, \alpha, \gamma)$ is called $\omega$-weakly asymptotically abelian for $\omega \in \mathcal{S}_{\beta}(\A, \alpha)$, if for all $a, b, c \in \A$
	\begin{equation*}
		\omega(a \comm{\alpha_t(b)}{c}) \xrightarrow[t \to \infty]{} 0.
	\end{equation*}
	It is similarly called $\gamma$-graded $\omega$-weakly asymptotically abelian for $\omega \in \mathcal{S}_{\beta}(\A, \alpha, \gamma)$, if
	\begin{equation*}
	\omega(a \comm{\alpha_t(b)}{c}_\gamma) \xrightarrow[t \to \infty]{} 0, \quad \text{where} \quad\comm{b}{c}_\gamma = bc - \gamma(c)b.
	\end{equation*}

	A graded $W^*$-dynamical system $(\M, \Om, V)$ is called asymptotically abelian, if for all $x, y \in \M$ 
	\begin{equation*}
		\comm{\alpha_t(x)}{y} \xrightarrow[t \to \infty]{\text{WOT}} 0,
		\end{equation*}
	where $\alpha$ is the modular flow associated to $\Om$. It is similarly called $V$-graded asymptotically abelian, if
	\begin{equation*}
	\comm{\alpha_t(x)}{y}_V \xrightarrow[t \to \infty]{\text{WOT}} 0, \quad \text{where} \quad \comm{x}{y}_V = xy - \Ad_V(y)x.
	\end{equation*}
\end{definition}

A routine calculation then shows that, if $\A$ is (graded) $\omega$-weakly asymptotically abelian for $\omega$, then the associated von Neumann algebra $\M$ is (graded) asymptotically abelian for the modular flow. For a discussion of graded a\-symp\-totic abelianess, see \cite[Section~3]{BuchholzLongo:1999}. 

We now derive strong implications of (graded) asymptotic abelian dynamics for the twisted center.

\begin{lemma} \label{lem:W*AlgCentralizer}
	Let $(\M, \Om, V)$ be a graded von Neumann algebra. Then the following statements holds:
	\begin{enumerate}
		\item If $\M$ is asymptotically abelian, then $\Zen_\Omega(\M) = \Zen(\M)$ and in particular $\Zen(\M,V) \subset \Zen(\M)$; \label{item:AsymptoticallyAbelian}
		\item If $\M$ is $V$-graded asymptotically abelian, then $\Zen(\M,V) = \Zen_\Omega(\M) \subset \Zen(\M)$; \label{item:GradedAsymptoticallyAbelian}
		\item If $\Zen(\M, V) \subset \Zen(\M)$ and $V|_{p\Hil} \neq 1_{p \Hil}$ for all nonzero central projections $p \in \Zen(\M)^\gamma$, then $\Zen(\M, V) = \{0\}$; \label{item:Central}
		\item If $\M$ is a factor, $V \neq 1$ and $\M$ is either asymptotically abelian or $V$-graded asymptotically abelian, then $\Zen(\M, V) = \{0\}$.
	\end{enumerate}
\end{lemma}

\begin{proof}
	a) The inclusion $\Zen(\M) \subset \Zen_\Omega(\M)$ is evident. Thus, we show the opposite inclusion. Consider $x \in \Zen_\Omega(\M) = \M^\alpha$ and $y \in \M$, then
	\begin{equation}
	0 \xleftarrow[t \to \infty]{} \comm{\alpha_t(x)}{y} = \comm{x}{y}.
	\end{equation}
	The right hand side is independent of $t$ and thus $x \in \mathcal{Z}(\M)$. The inclusion $\Zen(\M,V) \subset \Zen(\M)$ follows from Lemma \ref{lem:TwistedCenterProperties}~\ref{item:inclusion}.
	
	b) The inclusion $\Zen(\M,V) \subset \Zen_\Omega(\M)$ is clear. Let $x \in \Zen_\Omega(\M)$ and $y \in \M$.  A similar calculation shows $\comm{x}{y}_\gamma = \comm{\alpha_t(x)}{y}_\gamma \xrightarrow[t \to \infty]{} 0$ and $x \in \Zen(\M,V)$.
	
	c) By Theorem \ref{thm:AbstractGNS_Split}, $\Zen(\M,V)$ splits into
	\begin{equation*}
	\Zen(\M, V) = \Zen(\M_p, V_p) \op \Zen(\M_{p^\perp}, V_{p^\perp}) = u \cdot \Zen(\M_p) \op \{0\},
	\end{equation*}
	with a partial isometry $u \in \Zen(\M, V)$ with $u^*u = p \in \Zen(\M)^\gamma$, and $V_p =uJuJ$. Then $u \in \Zen(\M, V) \cap \Zen(\M)$ implies $V_p = uJuJ = Ju^*JJuJ= p$. The assumption $V_p \neq p$ for $p \neq 0$ implies $p = 0$ and $\Zen(\M,V) =\{ 0 \}$.
	
	d) This follows directly by a combination of \ref{item:AsymptoticallyAbelian}, \ref{item:GradedAsymptoticallyAbelian} and \ref{item:Central}.
\end{proof}

These results can again be applied to graded $C^*$-dynamical systems with (graded) weakly asymptotically abelian dynamics. In particular, the following corollary holds by a combination of Theorem \ref{thm:KMSStateCrossedProduct}, \ref{thm:TwistedKMSFunctional} and Lemma \ref{lem:W*AlgCentralizer}.

\begin{corollary}\label{cor:asymptoticallyabelian}
	Let $(\A, \alpha, \gamma)$ be a graded $C^*$-dynamical system and $\om$ be an extremal KMS state in $\mathcal{S}_{\beta}(\A, \alpha, \gamma)$. If $\gamma \neq 1$ and $\A$ is either weakly asymptotically abelian or $\gamma$-graded weakly asymptotically abelian, then $\mathcal{F}_{\beta}(\A, \alpha, \gamma, \omega) = \{0\}$ and $\bomcan$ is the unique extension of $\om$ to $\bA$.
\end{corollary}

\section{KMS states on $\Zl_2$-crossed products of CAR algebras}\label{section:CAR}

\subsection{The {\rm CAR}-dynamical systems} \label{subsec:CAR_DynamicalSystem}

We now consider a particular $C^*$-algebra $\A$, namely the CAR algebra over a Hilbert space $\Hil$. Recall that $\CAR$ is the unique and simple $C^*$-algebra generated by a unit $1$ and elements $a(\varphi)$, $\varphi\in\Hil$, subject to the relations
\begin{align*}
	\Hil\ni\phi &\mapsto a^*(\phi) \, \, \text{is linear}, \\
	\acomm{a^*(\phi)}{a(\psi)} &= \scp{\psi}{\phi} \cdot 1 , \quad \phi, \psi \in \Hil, \\
	\acomm{a(\phi)}{a(\psi)} &= 0 = \acomm{ a^*(\phi)}{a^*(\psi)}, \quad \phi, \psi \in \Hil,
\end{align*}
where $\{\,\cdot\,,\,\cdot\,\}$ denotes the anticommutator \cite[Theorem~5.2.5]{BratteliRobinson:1997}. The norm on $\CAR$ satisfies $\|a(\varphi)\|=\|\varphi\|$ for all $\varphi\in\Hil$. Note that $\CAR$ is also generated by the field operators
\begin{align}\label{def:Phi}
	\Phi(\xi) &:=
	a^*(\xi)+a(\xi),\qquad\xi\in\Hil,
\end{align}
which depend real linearly on $\xi$, and satisfy
\begin{align}\label{eq:CAR}
	\{\Phi(\xi),\Phi(\xi')\} &= 2\,\Re\langle\xi,\xi'\rangle\cdot1.
\end{align}

We will here consider the case that the dynamics and grading are given by Bogoliubov automorphisms. That is, we consider a strongly continuous unitary one-parameter group $U(t)=e^{it\Ham}$ and a unitary selfadjoint grading operator $G$ on $\Hil$, and the unique dynamics and grading on $\CAR$ given by
\begin{align}
	\alpha^H_t(a(\varphi)) &= a(U(t)\varphi),
	\qquad
	\gamma_G(a(\varphi)) = a(G\varphi),\qquad \varphi\in\Hil.
\end{align}
We furthermore assume that $U(t)$ commutes with $G$, so that we obtain a graded $C^*$-dynamical system $\CARDyng$ in the sense of Definition~\ref{def:GradedDynamicalSystem}. Note that $\CARDyng$ is concretely represented on the Fermionic Fock space $\Fock$ over $\Hil$, with the dynamics and grading implemented as
\begin{align}
	\alpha^H_t(A) &= \Gamma(U(t))A\Gamma(U(t))^*
	,\qquad
	\gamma_G(A) = \Gamma(G)A\Gamma(G)^*, \qquad A \in \CAR.
\end{align}
Here $\Gamma(V)$ denotes the second quantization of a unitary $V$ on $\Hil$, namely the restriction of $\bigoplus_{n}V\tp{n}$ to $\Fock$. The grading given by $G=-1$ will be referred to as the canonical grading $\te:=\gamma_{-1}$. It coincides with $\Ad_{(-1)^N}$, where $N$ is the particle number operator of $\Fock$.

In the following, we want to apply the general results from Section~\ref{section:abstract} to $\CARDyng$. To begin with, we recall the structure of quasifree functionals on $\CAR$ and in particular the structure of the KMS states of $\CARDyn$. As a domain for potentially unbounded functionals, we define $\CAR_0\subset\CAR$ as the ${}^*$-algebra generated by the field operators $\Phi(\xi),\xi\in\Hil$.

A hermitian functional $\mu$ on $\CAR_0$ is called {\em quasifree} if $\mu(1)=1$,
\begin{align*}
	\mu(\Phi(\xi_1)\cdots\Phi(\xi_{2n-1})) &= 0,\\
	\mu(\Phi(\xi_1)\cdots\Phi(\xi_{2n})) &= \sum_{\pi\in{\rm P}_{2n}}
	{\rm sign}(\pi)\,\prod_{k=1}^n\mu(\Phi(\xi_{\pi(2k-1)})\Phi(\xi_{\pi(2k)})),
\end{align*}
for all $n\in\Nl$, $\xi_1,\ldots,\xi_n\in\Hil$, where $${\rm P}_{2n}=\{\pi\in S_{2n}\,:\,\pi(2j-1)<\pi(2j),\,\pi(2j-1)<\pi(2j+1),\;1\leq j\leq n\}$$ is the set of all partitions of $\{1,\ldots,2n\}$ into pairs (pairings). Clearly, a quasifree functional is completely determined by its two point function.

When $\mu$ is continuous, we denote its extension to the $C^*$-algebra $\CAR$ by the same symbol and call it a quasifree functional on $\CAR$. This is in particular the case when $\mu$ is positive.

\begin{proposition}
	\leavevmode
	\begin{enumerate}
		\item The set of all quasifree states $\mu$ on $\CAR$ is in bijection with the set of all anti selfadjoint operators $T=-T^*\in\B(\Hil)$ satisfying $\|T\|\leq1$. The state $\mu_T$ corresponding to $T$ has the two-point function
		\begin{align}\label{eq:covariance}
			\mu_T(\Phi(\xi)\Phi(\eta)) &= \Re\langle\xi,\eta\rangle+i\Re\langle\xi,T\eta\rangle.
		\end{align}
		\item Given any anti selfadjoint bounded operator $T=-T^*\in\B(\Hil)$, there exists a unique hermitian quasifree functional $\mu_T$ on $\CAR_0$ which satisfies \eqref{eq:covariance} and $\mu_T(1)=1$.
		\item The functional $\mu_T$ is invariant under $\gamma_G$ if and only if $[G,T]=0$.
	\end{enumerate}
\end{proposition}
\begin{proof}
	a) is well known (see, e.g., \cite{BalslevManuceauVerbeure:1968,DerezinskiGerard:2022}).
	
	b) Let $T$ be given and define $\mu_T:\CAR_0\to\Cl$ by \eqref{eq:covariance}, $\mu_T(1)=1$, and requiring it to be quasifree. This is a consistent definition resulting in a hermitian functional because the two-point function \eqref{eq:covariance} satisfies $\mu_T(\{\Phi(\xi),\Phi(\eta)\})=2\Re\langle\xi,\eta\rangle$ and $\overline{\mu_T(\Phi(\xi)\Phi(\eta))}=\mu_T(\Phi(\eta)\Phi(\xi))$. On the other hand, $T$ is uniquely determined by $\mu_T$.
	
	c) is clear.
\end{proof}

The following result can be found in \cite{BratteliRobinson:1997,DerezinskiGerard:2022}, see also \cite{Araki:1971} for the corresponding version on the self-dual CAR algebra.

\begin{proposition} \label{prop:ExistenceKMS}
	For every inverse temperature $\beta$ there exists a unique KMS state $\omega_{\beta}$ on $({\rm CAR}(\Hil),\alpha^H)$, namely
	\begin{align}\label{eq:QuasiFreeKMS}
		\om_\beta
		=
		\mu_{T_\beta},\qquad T_\beta = -i\,\tanh\left(\tfrac{\beta }{2}H\right).
	\end{align}
\end{proposition}

Note that in the tracial case $H=0$ this describes the unique tracial state $\mu_0$ on $\CAR$.

\subsection{KMS states on the CAR system}\label{section:KMS-CAR}

We now investigate the extension problem for the state $\om_\beta$ described in Proposition~\ref{prop:ExistenceKMS} to the crossed product of $({\rm CAR}(\Hil),\alpha^H)$ by the $\Zl_2$-action $\gamma_G$.

To describe this situation, we note that as a consequence of $U(t)=e^{itH}$ commuting with $G$, the underlying Hilbert space splits as a direct sum
\begin{align}
	\Hil &= \Hil_{\rm ev}\oplus \Hil_{\rm odd},
	\\
	G &= (+1) \oplus (-1),\\
	U(t) &= U_{\rm ev}(t)\oplus U_{\rm odd}(t),
\end{align}
with two unitary one-parameter groups $U_{\rm ev}(t)=e^{it\Ham_{\rm ev}}$ on $\Hil_{\rm ev}$ and $U_{\rm odd}(t)=e^{it\Ham_{\rm odd}}$ on $\Hil_{\rm odd}$.

As mentioned before, the grading $\gamma_G$ is implemented by $\Gamma(G)$ on $\Fock$. The operator $\Gamma(G)$ lies in ${\rm CAR}(\Hil)$ if and only if $\Hil_{\rm odd}$ is finite dimensional (see \cite{Araki:1968} for a self-dual version of this statement, or the proof of Proposition~\ref{lem:TwistedKMS-NonTrace} below). As $\omega_\beta$ is the unique KMS state on $({\rm CAR}(\Hil),\alpha^H)$, the GNS representation is a factor representation. We may therefore apply Lemma \ref{lem:TwistedCenterProperties} to the inner grading $\gamma = \Ad_{\Gamma(G)}$ to conclude that for $\dim\Hil_{\rm odd}<\infty$, there exist two extremal extensions $\bom_\beta^\pm$ of $\om_\beta$ to the crossed product, given by 
\begin{align}
	\bom_\beta^\pm\cong(\om_\beta,\pm\om_\beta(\,\cdot\;\Gamma(G)))
	,\qquad \dim\Hil_{\rm odd}<\infty,
\end{align}
in the notation of Theorem~\ref{thm:KMSStateCrossedProduct}. This case does need no further discussion.

In case $\dim\Hil_{\rm odd}=\infty$, we may apply Lemma~\ref{lemma:simple} because $\CAR$ is simple and get 
\begin{align}
	{\rm CAR}(\Hil) \rtimes_{\gamma_G} \Zlt\cong C^*(\CAR,\Gamma(G)).
\end{align}
According to Theorem~\ref{thm:KMSStateCrossedProduct}, the first step to understanding the KMS states of this $C^*$-algebra is to determine the $\gamma_G$-twisted KMS functionals $\rho$ of $({\rm CAR}(\Hil),\alpha^H)$. We start with the case without zero modes, i.e. $\ker H=\{0\}$. 

Twisted KMS functionals can well be unbounded, see \cite{BuchholzGrundling:2007_2,Hillier:2015} for results motivated by supersymmetry. We are however aiming for twisted KMS functionals that are continuous because they are dominated by $\om_\beta$. We therefore restrict already here to those functionals $\rho_\beta$ that at least have bounded $n$-point functions, i.e. for any $n\in\Nl$, there exists  $c_n>0$ such that 
\begin{align}
	|\rho_\beta(\Phi(\xi_1)\cdots\Phi(\xi_n))|\leq c_n\|\xi_1\|\cdots\|\xi_n\|.
\end{align}

\begin{proposition} \label{prop:GradedFunctional_NPointFunction}
	Consider the graded $C^*$-dynamical system $\CARDyng$ with $\dim\Hil_{\rm odd}=\infty$ and $\ker H=\{0\}$. Let $\rho_\beta$ be a $\gamma_G$-twisted hermitian KMS functional with bounded $n$-point functions.
	\begin{enumerate}
		\item If $0 \in \sigma(H_{\rm odd})$, then $\rho_\beta=0$.
		\item If $0 \notin \sigma(\Ham_{\rm odd})$, then 
		\begin{align}\label{eq:rhobetaquasifree}
			{\rho_\beta}|_{\CAR_0} = \rho_\beta(1)\cdot\mu_{T_\beta^{G}},\qquad 
			T_\beta^{G}=i\left(\frac{1-Ge^{\beta H}}{1+Ge^{\beta H}}\right).
		\end{align}
		In particular, $\rho_\beta$ is uniquely determined by its value $\rho_\beta(1)$.
	\end{enumerate}
\end{proposition}

\begin{proof}
	To improve readability of our formulae, we will sometimes use subscript notation $\Phi_\xi=\Phi(\xi)$, $a_\xi^*=a^*(\xi)$ etc. in this proof.
	
	We begin by calculating the one- and two-point function of $\rho_\beta$. For $\eta$ entire analytic for the one-parameter group $U(t)$ (and hence $a^*_\eta$ entire analytic for $\alpha^H$), the $\gamma_G$-twisted KMS condition requires $\rho_\beta(a^*(e^{-\beta H}\eta))=\rho_\beta(1\cdot a^*(e^{-\beta H}\eta))=\rho_\beta(a^*(\eta))$ and hence ${\rho_\beta(a^*((1-e^{-\beta H})\eta))=0}$. Since $\Ham$ is injective and $\rho_\beta$ has a bounded one-point function, this implies that $\rho_\beta(a^*(\psi))=0$ for all $\psi\in\Hil$. Taking into account that $\rho_\beta$ is hermitian, we arrive at $\rho_\beta(\Phi(\psi))=0$ for all $\psi\in\Hil$.
	
	For the two-point function, the $\gamma_G$-twisted KMS condition and $\gamma_G$-invariance of $\rho_\beta$ yield
	\begin{align*}
		\rho_\beta(a^*(\xi)a^*(Ge^{-\beta H}\eta))
		&=
		\rho_\beta(a^*(\eta)a^*(\xi))
		=
		-\rho_\beta(a^*(\xi)a^*(\eta)),\\
		\rho_\beta(a(\xi)a^*(Ge^{-\beta H}\eta))
		&=
		\rho_\beta(a^*(\eta)a(\xi))
		=
		-\rho_\beta(a(\xi)a^*(\eta))+\langle\xi,\eta\rangle\cdot\rho_\beta(1).
	\end{align*}
	
	The first equation implies $\rho_\beta(a^*(\xi)a^*((1+Ge^{-\beta H})\eta))=0$. Note that the operators 
	\begin{align}
		L_G:=(1+ Ge^{-\beta H})^{-1}
	\end{align}
	 are well-defined because $\ker H=\{0\}$, $G$ commutes with $H$, and $\sigma(G)\subset\{1,-1\}$. Hence, analogous to the one-point function, $\rho_\beta(a^*(\xi)a^*(\psi))=0$ for analytic $\xi,\psi$ with $\psi\in\dom(L_G)$. As these vectors run over dense subspaces and $\rho_\beta$ has a bounded two-point function, we conclude  $\rho_\beta(a^*(\xi)a^*(\psi))=\rho_\beta(a(\xi)a(\psi))=0$ for all $\xi,\psi\in\Hil$.

	The second equation implies $\rho_\beta(a(\xi)a^*(\psi))=\rho_\beta(1)\cdot\langle\xi,L_G\psi\rangle$ for any analytic vectors $\xi,\psi$ with $\psi\in\dom(L_G)$.  Analogously, one obtains 
	\begin{equation*}
		\rho_\beta(a^*(\xi)a(\psi))=\rho_\beta(1)\cdot\langle\psi,(1+Ge^{\beta H})^{-1}\xi\rangle.
	\end{equation*}
	Here the sign of $H$ is flipped because $a$ depends antilinearly on its argument, so that $z\mapsto a(e^{i\overline z H}\xi)$ is analytic.
	
	Collecting all information, the two-point function of $\rho_\beta$ is
	\begin{align*}
		\rho_\beta(\Phi(\xi)\Phi(\psi))
		&=
		\rho_\beta(1)\cdot\left(\langle\xi,(1+Ge^{-\beta H})^{-1}\psi\rangle+\langle\psi,(1+Ge^{\beta H})^{-1}\xi\rangle\right)
		\\
		&=\rho_\beta(1)\left(\Re\langle\xi,\psi\rangle+i\Re\left\langle\xi,i\frac{1-Ge^{\beta H}}{1+Ge^{\beta H}}\psi\right\rangle\right)
		\\
		&=\rho_\beta(1)\cdot\mu_{T_\beta^{G}}(\Phi(\xi)\Phi(\psi)).
	\end{align*}
	We note that $T_\beta^{G}$ is unbounded if and only if $0\in\sigma(H_{\rm odd})$. Since $\rho_\beta$ has a bounded two-point function, we conclude $\rho_\beta(1)=0$ in this case.

	To complete the proof, we only need to show that $\rho_\beta$ is, up to the prefactor $\rho_\beta(1)$, quasifree. This argument proceeds along the same lines as in the untwisted case: Let $n \in \Nl$ and $\xi_1,\ldots,\xi_n\in\Hil$ be entire analytic. Then
	\begin{align*}
		\rho_\beta&(\Phi_{\xi_1}\cdots\Phi_{\xi_{n-1}}a^*(Ge^{-\beta H}\xi_n))
		=
		\rho_\beta(a^*(\xi_n)\Phi_{\xi_1}\cdots\Phi_{\xi_{n-1}})
		\\
		&=
		\sum_{k=1}^{n-1}(-1)^{k+1}\langle\xi_k,\xi_{n}\rangle\,\rho_\beta(\Phi_{\xi_1}\cdots\widehat{k}\cdots\Phi_{\xi_{n-1}})
		 +(-1)^{n-1}
		\rho_\beta(\Phi_{\xi_1}\cdots\Phi_{\xi_{n-1}}a^*_{\xi_n}),
	\end{align*}
	which yields the recursion relation
	\begin{align*}
		\rho_\beta(\Phi_{\xi_1}\cdots\Phi_{\xi_{n-1}}a^*_{\xi_n})
		=
		\sum_{k=1}^{n-1}(-1)^{k+n+1}\langle\xi_k,L_{\left((-1)^{n}G\right)}\xi_{n}\rangle\,\rho_\beta(\Phi_{\xi_1}\cdots\widehat{k}\cdots\Phi_{\xi_{n-1}}),
	\end{align*}
	where we have used that $L_{-G}=(1-Ge^{\beta H})^{-1}$.
	
	Proceeding analogously for $a(\xi_n)$ instead of $a^*(\xi_n)$, one arrives at
	\begin{align*}
		\rho_\beta(\Phi_{\xi_1}\cdots\Phi_{\xi_n})
		=
		\sum_{k=1}^{n-1}(-1)^{k+n+1}
		\mu_{T_\beta^{\left((-1)^{n}G\right)}}(\Phi_{\xi_k}\Phi_{\xi_{n}})\,
		\rho_\beta(\Phi_{\xi_1}\cdots\widehat{k}\cdots\Phi_{\xi_{n-1}}),
	\end{align*}
	from which we conclude that all odd-point functions vanish due to the fact that the one-point function vanishes.
	
	For even $n$, the above recursion relation is the one underlying the quasifree structure, which finishes the proof.
\end{proof}

We next consider the tracial case, $H=0$. Recall that $\mu_0$ is the unique tracial state on $\CAR$, and that a $\gamma_G$-twisted KMS functional $\tau$ satisfies the twisted trace condition $\tau(xy)=\tau(y\gamma_G(x))$.

\begin{proposition}\label{prop:tracialcase}
	Consider the graded $C^*$-algebra $({\rm CAR}(\Hil),\gamma_G)$ with \linebreak $\dim\Hil_{\rm odd}=\infty$. Then the only $\gamma_G$-twisted tracial functional $\tau$ which is dominated by $\mu_0$ is $\tau=0$. Hence, the trace $\mu_0$ of $\CAR$ has a unique tracial extension to the crossed product ${\rm CAR}(\Hil)\rtimes_{\gamma_G}\Zl_2$, namely $\boldsymbol{\tau}^{\rm can}$.
\end{proposition}
\begin{proof}
	We first note that the twisted tracial functional $\tau$ restricts to an untwisted tracial functional on the fixed point algebra ${\rm CAR}(\Hil_{\rm ev})$ of $\gamma_G$. As this $C^*$-algebra has the unique tracial state $\mu_0^{\rm ev}$ and $\mu_0^{\rm ev}\pm\tau$ are tracial and positive by the domination property, we see that the restriction of $\tau$ to ${\rm CAR}(\Hil_{\rm ev})$ is a multiple of $\mu_0^{\rm ev}$. As this state is quasifree, all $n$-point functions $\tau(\Phi(\eta_1)\cdots\Phi(\eta_n))$ with $n$ odd and $\eta_1,\ldots,\eta_n\in\Hil_{\rm ev}$ vanish.
	
	For odd $n$, we may therefore without loss of generality assume $\eta_n\in\Hil_{\rm odd}$ and consider $\tau(\Phi(\eta_1)\cdots\Phi(\eta_n))$. Similar to the previous proof, the twisted trace condition yields 
	\begin{align*}
		\tau(\Phi(\eta_1)&\cdots \Phi(\eta_n))
		=
		-
		\tau(\Phi(\eta_n)\Phi(\eta_1)\cdots\Phi(\eta_{n-1}))
		\\
		&=
		2\sum_{k=1}^{n-1} (-1)^k\Re\langle\eta_k,\eta_n\rangle\,\tau(\Phi(\eta_1)\cdots\widehat{k}\cdots\Phi(\eta_{n-1}))
		-\tau(\Phi(\eta_1)\cdots\Phi(\eta_n))
	\end{align*}
	and hence a recursion formula expressing $\tau(\Phi(\eta_1)\cdots\Phi(\eta_n))$ in terms of one point functions $\tau(\Phi(\xi))$. But for $\xi\in\Hil_{\rm ev}$, we have $\tau(\Phi(\xi))=0$ as pointed out before, and for $\xi\in\Hil_{\rm odd}$, we have $\tau(\Phi(\xi))=0$ by the $G$-invariance of $\tau$. Hence all $n$-point functions with odd $n$ vanish.
	
	For even $n$, let $\eta_1,\ldots,\eta_n\in\Hil$ and choose $\xi\in\Hil_{\rm odd}$ normalized and such that $\xi\perp{\rm span}\{\eta_1,\ldots,\eta_n\}$ (this is possible thanks to $\dim\Hil_{\rm odd}=\infty$). Then the twisted trace condition and the CAR relations imply
	\begin{align*}
		\tau(a(\xi)^*\Phi(\eta_1)&\cdots\Phi(\eta_n)a(\xi))
		=
		-\tau(a(\xi)a(\xi)^*\Phi(\eta_1)\cdots\Phi(\eta_n))
		\\
		&=-\tau(\Phi(\eta_1)\cdots\Phi(\eta_n))+\tau(a(\xi)^*\Phi(\eta_1)\cdots\Phi(\eta_n)a(\xi)).
	\end{align*}
	This yields $\tau(\Phi(\eta_1)\cdots\Phi(\eta_n))=0$ and in particular $\tau(1)=0$. Thus $\tau=0$.
	
	The claim about the unique extension of $\mu_0$ is now a direct consequence of the results of Theorem~\ref{thm:KMSStateCrossedProduct}.
\end{proof}

This proof demonstrates that the tracial case is considerably simpler than the general KMS case. In particular, the extension problem for traces is already settled at this point, whereas in the case of non-trivial twisted KMS functionals on $\CAR$, the crucial question is now whether they are dominated by the unique KMS state discussed before. In particular, we have not yet clarified under which conditions $\mu_{T_\beta^G}$ is continuous, which is a consequence of $\rho_\beta(1)\neq0$ and $\rho_\beta$ being dominated by $\om_\beta$. The product
\begin{align}\label{eq:CH}
	c_{\beta H}
	:=
	\prod_{\la\in\sigma(H_{\rm odd})}\frac{1 - e^{-\abs{\beta \lambda}}}{1 + e^{-\abs{\beta \lambda}}}
\end{align}
over the spectrum $\sigma(H_{\rm odd})$ (counting with multiplicity) of the $G$-odd part $H_{\rm odd}$ of the Hamiltonian will be relevant for this.

\begin{proposition} \label{lem:TwistedKMS-NonTrace}
	Consider the graded $C^*$-dynamical system $\CARDyng$, with $\dim\Hil_{\rm odd}=\infty$ and $\ker(H)=\{0\}$. The following are equivalent:
	\begin{enumerate}
		\item $\mathcal{F}_{\beta}(\CAR, \alpha^H, \gamma_G, \omega_\beta)\neq\{0\}$,
		\item $c_{\beta H}>0$,
		\item $\Tr_{\Hil_{\rm odd}} \left(e^{-|\beta H_{\rm odd}|}\right)<\infty$.
	\end{enumerate}
	In this case, $\mu_{T_\beta^G}$ is bounded, and 
	\begin{align}\label{eq:twistedmultiples}
		\mathcal{F}_{\beta}(\CAR, \alpha^H, \gamma_G, \omega_\beta)
		&=
		\{c\cdot\mu_{T_\beta^{G}}\,:\,c\in[-c_{\beta H},c_{\beta H}]\}.
	\end{align}
\end{proposition}

\begin{proof}
	a) $\Rightarrow$ b): Let $\rho_\beta\in\mathcal{F}_{\beta}(\CAR, \alpha^H, \gamma_G, \omega_\beta)\neq\{0\}$, $n\in\Nl$ and $\xi_1,\ldots,\xi_n\in\Hil_{\rm odd}$ be mutually orthogonal unit vectors such that $$\langle\xi_j,(1+ e^{-\beta H})^{-1}\xi_k\rangle=0=\langle\xi_j,(1- e^{-\beta H})^{-1}\xi_k\rangle,   \quad \forall j\neq k.$$Such vectors exist because of $\dim\Hil_{\rm odd}=\infty$. In view of the quasifree structure of $\om_\beta$ and $\rho_\beta$, the operator  $A:=a(\xi_1)\cdots a(\xi_n)$ satisfies
	\begin{align*}
		\rho_\beta(A^*A)
		&=
		\rho_\beta(1)\prod_{k=1}^n\rho_\beta(a(\xi_k)^*a(\xi_k))
		=
		\rho_\beta(1)\prod_{k=1}^n\langle \xi_k,(1-e^{\beta H})^{-1}\xi_k\rangle
		,\\
		\rho_\beta(AA^*)
		&=
		\rho_\beta(1)\prod_{k=1}^n\langle \xi_k,(1-e^{-\beta H})^{-1}\xi_k\rangle
		,\\
		\om_\beta(A^*A)
		&=
		\prod_{k=1}^n\om_\beta(a(\xi_k)^*a(\xi_k))
		=
		\prod_{k=1}^n\langle \xi_k,(1+e^{\beta H})^{-1}\xi_k\rangle
		,\\
		\om_\beta(AA^*)
		&=
		\prod_{k=1}^n\langle \xi_k,(1+e^{-\beta H})^{-1}\xi_k\rangle
		.
	\end{align*}
	As $\om_\beta$ dominates $\rho_\beta$, we have $|\rho_\beta(A^*A)|\leq\om_\beta(A^*A)$ and $|\rho_\beta(AA^*)|\leq\om_\beta(AA^*)$, i.e.
	\begin{align}\label{eq:specest}
		|\rho_\beta(1)|
		\leq
		\prod_{k=1}^n \frac{\langle \xi_k,(1+e^{\pm\beta H})^{-1}\xi_k\rangle}{|\langle \xi_k,(1-e^{\pm\beta H})^{-1}\xi_k\rangle|}
		\leq
		\prod_{k=1}^n \frac{1}{|\langle \xi_k,(1-e^{\pm\beta H})^{-1}\xi_k\rangle|}
		,
	\end{align}
	provided $\langle \xi_k,(1-e^{\pm\beta H})^{-1}\xi_k\rangle\neq0$. 
	
	These inequalities hold for both signs $\pm$ and can be used to obtain spectral information about the Hamiltonian $\beta H$. We denote its spectral projections by $E^{\beta H}$. Suppose that there exists $r>0$ such that $\Hil_{\rm odd}^{r,+}:=E^{\beta H}([0,r])\Hil_{\rm odd}$ is infinite-dimensional. Then, for any $n\in\Nl$, we may choose the vectors $\xi_1,\ldots,\xi_n\in\Hil_{\rm odd}^{r,+}$, and obtain the estimate $$|\langle\xi_k,(1-e^{-\beta H})^{-1}\xi_k\rangle|=\langle\xi_k,(1-e^{-\beta H})^{-1}\xi_k\rangle\geq(1-e^{-r})^{-1}.$$
	Inserting this into the rightmost term in \eqref{eq:specest} with the lower sign now yields $|\rho_\beta(1)|\leq(1-e^{-r})^{n}$. As this estimate holds for all $n$, we conclude $\rho_\beta(1)=0$.
	
	An analogous argument works in case $\Hil_{\rm odd}^{r,-}:=E^{\beta H}([-r,0])\Hil_{\rm odd}$ is infinite-dimensional. In this case one estimates $$|\langle\xi_k,(1-e^{\beta H})^{-1}\xi_k\rangle|=\langle\xi_k,(1-e^{\beta H})^{-1}\xi_k\rangle\geq(1-e^{-r})^{-1}.$$Inserting this into \eqref{eq:specest} with the upper sign again results in $\rho_\beta(1)=0$.
	
	As $\rho_\beta(1)=0$ implies the contradiction $\rho_\beta=0$, we see that $E^{\beta H}([-r,r])\Hil_{\rm odd}$ is finite-dimensional for all $r>0$. This argument shows already that $\sigma(H_{\rm odd})$ consists only of eigenvalues with finite multiplicity and no finite accumulation point. Let us denote the eigenvalues $(\la_k)_{k\in\Nl}$ (repeated according to multiplicity) and choose the $\xi_k$ to be normalized corresponding eigenvectors of $H_{\rm odd}$. Then the first inequality in \eqref{eq:specest} yields

	\begin{equation} \label{eq:ProdIneq}
	0<\abs{\rho_\beta(1)} \leq \prod_{k=1}^{n} \abs{\frac{1 -e^{\pm\beta \lambda_k}}{1 + e^{\pm\beta \lambda_k}}} = \prod_{k=1}^{n} \frac{1 - e^{-\abs{\beta \lambda_k}}}{1 + e^{-\abs{\beta \lambda_k}}}
	\end{equation}
	for all $n\in\Nl$. Hence the infinite product $\prod_{k=1}^{\infty} \frac{1 - e^{-\abs{\beta \lambda_k}}}{1 + e^{-\abs{\beta \lambda_k}}}$ does not converge to $0$, i.e. $c_{\beta H}>0$.
	
	b) $\Leftrightarrow$ c) By standard estimates translating infinite products into infinite sums, one obtains $c_{\beta H}>0\Leftrightarrow \Tr_{\Hil_{\rm odd}} \left(e^{-|\beta H_{\rm odd}|}\right)=\sum_{k=1}^\infty e^{-|\beta \lambda_k|}<\infty$.
	
	c) $\Rightarrow$ a): The idea of this part of the proof is to use Proposition~\ref{prop:weaklyinnergrading}. We choose an orthonormal basis $(\varphi_k)$ of $H_{\rm odd}$ with corresponding eigenvalues $\la_k$ and consider the operators $P_k := a^*(\phi_k)a(\phi_k)$. The anticommutation relations imply that the $P_k$ are mutually commuting orthogonal projections satisfying 
	\begin{equation*}
		e^{i \pi P_k} = 1 - 2P_k = i \Phi(i \phi_k)\Phi(\phi_k).
	\end{equation*}
	The unitaries implementing the weakly inner grading from Proposition~\ref{prop:weaklyinnergrading} are here taken as
	\begin{align*}
		u_n := \prod_{k=1}^n \sgn(\beta\la_k)e^{i\pi P_k}.
	\end{align*}
	They are clearly selfadjoint as well as invariant under the grading and under the dynamics. In view of the quasifree structure of $\om_\beta$, we have 
	\begin{align*}
		\om_\beta(u_n) 
		&=
		i^n\om_\beta\left(\prod_{k=1}^n\sgn(\beta\la_k)\Phi(i\varphi_k)\Phi(\varphi_k)\right)
		\\
		&=
		i^n\prod_{k=1}^n\sgn(\beta\la_k)\om_\beta\left(\Phi(i\varphi_k)\Phi(\varphi_k)\right)
		=
		\prod_{k=1}^n\tanh\frac{|\beta\la_k|}{2}.
	\end{align*}
	This product converges to the finite non-zero value $c_{\beta H}$, as verified in b) $\Leftrightarrow$ c). It is then clear that for $n<m$, we have $\om_\beta(u_nu_m)=\prod_{k=n+1}^m\tanh\frac{|\beta\la_k|}{2}\to1$ as $n,m\to\infty$, verifying another assumption of Proposition~\ref{prop:weaklyinnergrading}.
	
	It remains to show that $\Ad_{u_n}$ approximates the grading. As a suitable subalgebra, we take $\A_0$ to be the ${}^*$-algebra generated by the $\Phi(\varphi_k)$, $k\in\Nl$, and $\Phi(\xi)$, $\xi\in\Hil_{\rm ev}$ analytic for $U_{\rm ev}$. This algebra satisfies the assumptions of Proposition~\ref{prop:weaklyinnergrading}. Moreover, we have 
	\begin{align}
		u_n\Phi(\varphi_m)u_n
		&=
		\begin{cases}
			-\Phi(\varphi_m) & m\leq n\\
			+\Phi(\varphi_m) & m>n
		\end{cases}
		,
		\qquad 
		u_n\Phi(\xi)u_n = \Phi(\xi),\quad \xi\in\Hil_{\rm ev}.
	\end{align}
	This yields $\|\Ad_{u_n}(a)-\gamma_G(a)\|\to0$ for any $a\in\A_0$.
	
	We may therefore apply Proposition~\ref{prop:weaklyinnergrading} to conclude that \linebreak $\rho_\beta(a)  :=\lim_n\om_\beta(au_n)$ lies in $\mathcal{F}_{\beta}(\CAR, \alpha^H, \gamma_G, \omega_\beta)$. As $\rho_\beta(1)=c_{\beta H}\neq0$, this proves a). 
	
	By application of Proposition \ref{prop:GradedFunctional_NPointFunction}, we also get $\rho_\beta=\rho_\beta(1)\,\mu_{T_\beta^G}$. According to the estimate \eqref{eq:ProdIneq}, $c_{\beta H}$ is an upper bound for $|\rho_\beta(1)|$ for any $\rho_\beta\in \mathcal{F}_{\beta}(\CAR, \alpha^H, \gamma_G, \omega_\beta)$. Therefore this upper bound is attained, and \eqref{eq:twistedmultiples} follows.
\end{proof}

In the self-dual setting, Hillier studied boundedness questions for twisted KMS-functionals in \cite{Hillier:2015}. His Theorem A.4 appears to be related to some parts of our analysis above. As it is formulated in the context of the self-dual algebra, we however preferred to work with the direct argument given above.

\medskip 

We do not need a concrete realization of the GNS representation of $\om_\beta$ for our analysis. For completeness, let us still point out that this representation can be realized as (see  \cite{ArakiWyss:1964, Araki:1971, BaumgartelJurkeLledo:2002}) as
\begin{align*}
\Fock_\beta& = \Fock \otimes \overline{\Fock}, \quad \Om_\beta = \Om \ot \overline{\Om}, \quad V= \Gamma(G) \ot \overline{\Gamma(G)}, \\
a_\beta^*(\phi) & :=  \pi_\beta(a^*(\phi)) = a^*(\sqrt{T}\phi)\otimes \overline{1} + (-1)^N \otimes \overline{a(\sqrt{1 -T}\phi)}, \\
a_\beta(\phi) & := \pi_\beta(a(\phi)) = a(\sqrt{T}\phi) \otimes \overline{1} + (-1)^N\otimes \overline{a^*(\sqrt{1-T}\phi)},
\end{align*}
where $T = (1 + e^{- \beta \Ham})^{-1}$ and $\overline{\Fock}$ denotes the complex conjugate Hilbert space of $\Fock$. The modular conjugation of the enveloping von Neumann algebra $\M$ is given by $J = [(-1)^\frac{N(N-1)}{2}\ot \overline{(-1)^\frac{N(N-1)}{2}}]F$, where $F$ denotes the tensor flip. In the notation of the proof of Proposition~\ref{lem:TwistedKMS-NonTrace}, the product
\begin{equation*}
	R := \prod_{k=1}^\infty\sgn(\beta \lambda_k) e^{i \pi a_\beta^*(\phi_k) a_\beta(\phi_k)}
\end{equation*}
converges weakly. It is selfadjoint, unitary and $R \in \Zen(\M, V)$ if and only if $0 \notin \sigma(\Ham_{\rm odd})$ and $\Tr_{\Hil_{\rm odd}} \left(e^{-|\beta H_{\rm odd}|}\right)<\infty$ and vanishes otherwise. As the KMS state $\omega_{\beta}$ is unique, $\M$ is factor and therefore, Corollary~\ref{cor:one-dimensional} can be applied. This shows  $\Zen(\M,V) = R \cdot \Cl$.

We summarize our results on the CAR system in the following theorem. The proof follows directly from Proposition~\ref{lem:TwistedKMS-NonTrace} in connection with Theorem~\ref{thm:KMSStateCrossedProduct}.

\begin{theorem} \label{thm:CrossedProduct_CAR_KMSStates}
	The crossed product of the graded $C^*$-dynamical system \linebreak $\CARDyng$ (with $\ker(H)=\{0\}$ and $\dim\Hil_{\rm odd}=\infty$) has a unique KMS state at inverse temperature $\beta$, namely the canonical extension $\bomcanb$ of $\om_\beta$ \eqref{eq:QuasiFreeKMS}, if and only if the Gibbs type condition
	\begin{align}\label{eq:GibbsType}
		\Tr_{\Hil_{\rm odd}}(e^{-|\beta H_{\rm odd}|})<\infty
	\end{align}
	is violated. If \eqref{eq:GibbsType} holds, the crossed product has two extremal KMS states given via the bijection in Theorem~\ref{thm:KMSStateCrossedProduct} by
	\begin{align}
		\bom^{\pm}_\beta = (\bomcanb,\pm c_{\beta H}\cdot\mu_{T_\beta^G}),
	\end{align}
	with $\mu_{T_\beta^G}$ the quasifree functional \eqref{eq:rhobetaquasifree} and $c_{\beta H}>0$ the number defined in \eqref{eq:CH}.
\end{theorem}

Observe that in case $\beta H_{\rm odd}$ is bounded below and the grading is the canonical one, $G=-1$, \eqref{eq:GibbsType} implies the actual Gibbs condition $\Tr_{\Hil}(e^{-\beta H})<\infty$. Hence in this case, the second quantized Hamiltonian also satisfies the Gibbs condition, namely $\Tr_{\mathcal{F}_-(\Hil)}(e^{-\beta d\Gamma(H)})<\infty$ \cite[Proposition~5.2.22]{BratteliRobinson:1997}. We can therefore use the same idea as in Lemma~\ref{lemma:Gibbs} and obtain a Gibbs extension $\bomGibbsb$ of $\bom_\beta$ to the crossed product. Since
\begin{align*}
	\bomGibbsb((-1)^N)
	=
	\prod_{\la\in\sigma(H)} \frac{1 - e^{- \beta \la}}{1+ e^{- \beta \la}}
	=
	\bom^+_\beta((-1)^N)
	,
\end{align*}
this extension coincides with $\bom^+_\beta$.

\medskip

This concludes our analysis of the KMS states of the crossed product for the CAR system. We have not discussed the most general case of non-trivial kernel $\ker{H}\neq \{0\}$ which mixes the tracial case $H=0$ with the case $\ker H=\{0\}$. We expect that this can be done efficiently by splitting $\CAR$ into a graded tensor product \cite{CrismaleDuvenhageFidaleo:2021, ArakiMoriya:2003} with the factors corresponding to $H=0$ and $\ker{H}=\{0\}$, respectively.

\section{Examples from mathematical physics}\label{section:physics}

Crossed products appear in several places in mathematical quantum physics, for example in the context of chemical potential in gauge theories \cite{ArakiKastlerTakesakiHaag:1977}, or in the context of the lattice Ising model \cite{ArakiEvans:1983}. In the latter situation, one considers the CAR algebra for the lattice $\Zl$ (generated by $a^\#(\varphi_n)$, $n\in\Zl$) and the grading given by $G\varphi_n=\eps(n)\varphi_n$ with $\eps(n)=1$ for $n\geq1$ and $\eps(n)=-1$ otherwise. This is an example of a non-canonical grading $G$ as considered in the previous section.

In lattice or continuum models with positive Hamiltonian $H$ and positive inverse temperature $\beta>0$, the Gibbs type condition \eqref{eq:GibbsType} coincides with the usual Gibbs condition from statistical mechanics. By adding a potential term to the Hamiltonian with suitable growth rate at infinity, one then has many examples in which \eqref{eq:GibbsType} is satisfied, sometimes only for small enough $\beta$ (existence of maximal temperature). A detailed investigation of the KMS states of such models and their interpretation will appear elsewhere.

Here we restrict ourselves to consider another example, a relativistic quantum field theory on two-dimensional Minkowski space-time which can be described by $\Zl_2$-crossed products in a non-obvious manner. The model we are considering is often referred to as the Ising QFT. It arises from a scaling limit of the two-dimensional Ising lattice model \cite{McCoyTracyWu:1977, SatoMiwaJimbo:1978} and turns out to be a completely integrable QFT with two-particle S-matrix equal $-1$ and fits into a larger family of Fermionic models \cite{BostelmannCadamuro:2021}.

To see its relation with $\Zl_2$-crossed products, we will describe it from an operator algebraic perspective in which the connection with the Ising lattice model is no longer relevant. The Ising QFT belongs to a family of integrable QFTs that can be analyzed in the framework of algebraic quantum field theory in its vacuum representation (zero temperature). We refer to the review for detailed information about this construction \cite{Lechner:AQFT-book:2015}, and briefly explain the main points here.

The Hilbert space of the model is the Fermi Fock space $\Fock$ over a one particle space carrying the irreducible unitary positive energy representation $U$ of the proper orthochronous Poincaré group $\CP_0(2)$ in two dimensions with mass $m>0$ and spin zero, i.e. $\Hil=L^2(\Rl,d\te)$. With $p(\te):=m(\cosh\te,\sinh\te)$, the group elements $(x,\la)$ (with $x\in\Rl^2$ denoting space-time translations and $\la$ the rapidity parameter of a Lorentz boost) are represented as
\begin{align}
	(U(x,\la)\psi)(\te)=e^{ip(\te)\cdot x}\psi(\te-\la).
\end{align}
In particular, the Hamiltonian of this representation is
\begin{align}\label{eq:IsingHam}
	(H\psi)(\te)=m\cosh(\te)\cdot\psi(\te).
\end{align}
As before, we will write $\alpha^H$ to denote the dynamics given by adjoint action of the second quantization of $U((t,0),0)=e^{itH}$ on $\B(\Fock)$.

The construction of the QFT proceeds with the help of the fields $\Phi$ introduced in \eqref{def:Phi} and a localization structure. To describe it, we consider the map 
\begin{align}
 {\mathscr S}(\Rl^2) \ni f \longmapsto \hat f\in\Hil,\qquad \hat f(\te):=\tilde f(p(\te)),
\end{align}
where $\tilde f$ is the Fourier transform of $f$. For any open subset $\CO\subset\Rl^2$, one then forms the $C^*$-algebras 
\begin{align}
	\CC(\CO) := C^*(\Phi(\hat f)\,:\,f\in C_{c,\Rl}^\infty(\CO)) \subset \B(\Fock).
\end{align}
This construction provides us with a net of $C^*$-algebras transforming covariantly under the second quantization of $U$, namely $\CC(\CO_1)\subset\CC(\CO_2)$ for $\CO_1\subset\CO_2$, and $\Gamma(U(g))\CC(\CO)\Gamma(U(g^{-1}))=\CC(g\CO)$ for all $g\in\CP_0(2)$. Moreover, the Fock vacuum~$\Om$ is cyclic for $\CC(\CO)$ for all open non-empty $\CO$.

The net $\CO\mapsto\CC(\CO)$ is however not local, i.e. $\CC(\CO_1)$ and $\CC(\CO_2)$ do not commute for $\CO_1$ spacelike to $\CO_2$. This can be understood as a consequence of the Spin-Statistics Theorem, since the field~$\Phi$ is build from the CAR but $U$ has spin zero. The local content of the Ising QFT is uncovered by realizing a hidden locality property. Namely, for the particular Rindler wedge region ${\mathcal W}:=\{(x_0,x_1)\in\Rl^2\,:\,x_1>|x_0|\}$, one finds that the vacuum $\Om$ is also separating for $\CC({\mathcal W})$. Thus Tomita-Takesaki theory applies to the pair $(\CC({\mathcal W})'',\Om)$, and the modular conjugation $J_{\mathcal W}$ turns out to be given by \cite{BuchholzLechner:2004}
\begin{align*}
	J_{\mathcal W} = \Gamma(J)\,(-1)^{N(N-1)/2},
\end{align*}
where $J\xi=\overline\xi$ is pointwise complex conjugation on $\Hil=L^2(\Rl)$. Therefore the second field operator $\Phi'$, defined as
\begin{align}\label{eq:PhiPrime}
	\Phi'(\xi) := J_{\mathcal W}\Phi(J\xi)J_{\mathcal W} = (a^*(\xi)-a(\xi))\,(-1)^N,
\end{align}
generates the commutant $\CC({\mathcal W})'=\{\Phi'(\hat f)\,:\,f\in C_{c,\Rl}^\infty({\mathcal W}')\}$. Its commutation relations with $\Phi$ are given by
\begin{align}\label{eq:PhiPhiComm}
	[\Phi(\xi),\Phi'(\eta)]
	=
	2i\,\Im\langle\xi,\eta\rangle\,(-1)^N,
\end{align}
which vanishes for $\xi=\hat f$, $\eta=\hat g$ with the supports of $f$ and $g$ spacelike separated. These observations lead to a local QFT by assigning to a double cone $\CO_{x,y}:=(-{\mathcal W}+x)'\cap({\mathcal W}+y)'$ (with dashes denoting causal complements) the von Neumann algebra
\begin{align}
	\A(\CO_{xy})
	:=
	\CC'(-{\mathcal W}+x)'\cap\CC({\mathcal W}+y)'.
\end{align}
This defines a quantum field theory satisfying all axioms of quantum field theory in its vacuum representation, i.e. at temperature zero. In particular, the vacuum is cyclic and separating for each $\A(\CO_{xy})$, the proof of which relies on the split property \cite{Lechner:2005}. It is interacting with S-matrix $(-1)^{N(N-1)/2}$ and asymptotically complete \cite{Lechner:2008}.

To connect with crossed products, we consider the two global field algebras $\CC(\Rl^2)$ and $\CC'(\Rl^2)$ generated by the fields $\Phi$ and $\Phi'$, respectively, and the extended global field algebra
\begin{equation*}
	\widehat{\CC}:={C^*}(\CC(\Rl^2),\CC'(\Rl^2)),
\end{equation*}
which contains both fields $\Phi$ and $\Phi'$.

\begin{theorem}
	The field algebras of the two fields $\Phi,\Phi'$ defining the Ising QFT are $\CC(\Rl^2)=\CC'(\Rl^2)=\CAR$, and the extended field algebra is
	\begin{align}
		\widehat{\CC}=C^*(\CAR,(-1)^N)\cong{\rm CAR}(\Hil)\rtimes\Zl_2,
	\end{align}
	where the crossed product is taken w.r.t. the canonical grading $\Ad_{(-1)^N}$.

	At each inverse temperature $\beta\neq0$, the $C^*$-dynamical system $(\widehat{\CC},\alpha^H)$ consisting of the extended field algebra $\widehat\CC$ and the dynamics given by the one particle Hamiltonian $H$ \eqref{eq:IsingHam} has a unique KMS state.
\end{theorem}
\begin{proof}
	The inclusion $\CC(\Rl^2)\subset\CAR$ holds by definition. Passing to complex linear test functions and limits, one notes $a^*(\hat f)+a(\hat{\overline{f}})\in\CC(\Rl^2)$ for any (complex) $f\in{\mathscr S}(\Rl^2)$. Hence one can choose $f$ in such a way that $\hat{\overline f}=0$, and obtain a dense set of vectors $\hat f\in\Hil$ in this way. This shows the opposite inclusion $\CAR\subset\CC(\Rl^2)$.

	In view of the definition of $\Phi'$, we clearly have $\CC'(\Rl^2)\cong\CC(\Rl^2)$.

	The extended global field algebra $\widehat{\CC}$ contains $\CAR$ and, in view of \eqref{eq:PhiPhiComm}, also the grading operator $(-1)^N$. Hence $\widehat\CC=C^*(\CAR,(-1)^N)$.

	As $\dim\Hil=\infty$, this grading operator is not contained in $\CAR$ \cite{Araki:1971}. Taking into account that $\CAR$ is simple, we may apply Lemma~\ref{lemma:simple} and conclude $C^*(\CAR,(-1)^N)\cong{\rm CAR}(\Hil)\rtimes \Zl_2$.

	The statement about existence and uniqueness of KMS states follows by application of Theorem~\ref{thm:CrossedProduct_CAR_KMSStates}: The Hamiltonian $H$ \eqref{eq:IsingHam} has continuous spectrum and therefore violates the Gibbs type condition \eqref{eq:GibbsType}. Hence the unique $(\alpha^H,\beta)$-KMS state of $\CAR$ has a unique extension to $\widehat\CC$.
\end{proof}

We note that all $n$-point functions of the KMS state $\bom_\beta$ of $\widehat\CC$ (in both field types) can be read off from our construction in Section~\ref{section:CAR}. It should furthermore be noted that $\widehat\CC$ contains many even local observables, because the even part of $\A(\CO_{xy})$ is generated by even polynomials in the field $\Phi$ \cite{BuchholzSummers:2007}. Nonetheless, the field algebra $\widehat\CC$ differs from the quasilocal $C^*$-algebra $\mathfrak{A}$ of the Ising QFT in its odd elements, so that the KMS state constructed here describes only parts of the Ising QFT at finite temperature $\beta^{-1}>0$. In conclusion, we note that also the quasilocal $C^*$-algebra ${\mathfrak A}$ can be shown to have KMS states by combining the results of \cite{BuchholzJunglas:1989,Lechner:2005}. We leave a more detailed investigation of their relation to a future work.

\section*{Acknowledgments}

We would like to thank Kang Li for stimulating discussions about crossed products. Support by the German Research Foundation DFG through the Heisenberg project ``Quantum Fields and Operator Algebras'' (LE 2222/3-1) is gratefully acknowledged.

\end{document}